\documentclass[11pt]{amsart}
\usepackage{amscd,amssymb,url,colonequals,enumitem,xy}
\xyoption{all}
\newtheorem{thm}{Theorem}[section]
\newtheorem{prop}[thm]{Proposition}
\newtheorem{lem}[thm]{Lemma}
\newtheorem{cor}[thm]{Corollary}
\newtheorem{conj}[thm]{Conjecture}

\theoremstyle{definition}
\newtheorem{defn}[thm]{Definition}
\newtheorem{exa}[thm]{Example}

\theoremstyle{remark}
\newtheorem{rem}[thm]{Remark}

\newcommand{\op}{\operatorname}

\newcommand{\Gr}{\operatorname{Gr}}
\newcommand{\Char}{\op{char}}

\newcommand{\Span}{\op{Span}}
\newcommand{\Rep}{\op{Rep}}
\newcommand{\rank}{\op{rank}}
\newcommand{\Split}{\op{Split}_k}
\newcommand{\ed}{\op{ed}}
\newcommand{\Id}{\op{Id}}

\newcommand{\Spec}{\op{Spec}}
\newcommand{\trdeg}{\op{trdeg}}
\newcommand{\rk}{\op{rk}}
\newcommand{\GL}{\op{GL}}
\newcommand{\PGL}{\op{PGL}}

\newcommand{\Gal}{\op{Gal}}
\newcommand{\sep}[1]{#1_{\op{sep}}}
\newcommand{\alg}[1]{#1_{\op{alg}}}
\newcommand{\pcl}[1]{#1^{(p)}}
\newcommand{\Aut}{\op{Aut}}
\newcommand{\Br}{\op{Br}}
\newcommand{\Gm}{\op{\mathbb{G}}_m}
\newcommand{\Hom}{\op{Hom}}
\newcommand{\ind}{\op{ind}}
\newcommand{\Ex}{\op{Ex}}
\newcommand{\Ext}{\widetilde{\op{Ex}}}
\newcommand{\ZZ}{\mathbb{Z}}
\newcommand{\FF}{\mathbb{F}}
\newcommand{\Fields}{\op{Fields}}
\newcommand{\Sets}{\op{Sets}}
\newcommand{\NN}{\mathbb{N}}
\newcommand{\Diag}{\op{Diag}}
\newcommand{\QQ}{\mathbb{Q}}
\newcommand{\R}{\op{R}}
\newcommand{\LL}{(\mathbb{L})}
\newcommand{\Stab}{\op{Stab}}

\def\mathbi#1{\textbf{\em #1}}

\begin{document}
\title{Essential $\mathbi{p}$-dimension of algebraic tori}
\author[R. L\"otscher]{Roland L\"otscher$^{(1)}$}
\thanks{$^{(1)}$ Roland L\"otscher was partially supported by the Swiss National Science Foundation (Schweizerischer Nationalfonds).}
\author[M. M{\tiny{ac}}Donald]{Mark M{\tiny{ac}}Donald}
\author[A. Meyer]{Aurel Meyer$^{(2)}$}
\thanks{$^{(2)}$ Aurel Meyer was partially supported by
a University Graduate Fellowship at the University of British Columbia}
\author[Z. Reichstein]{Zinovy Reichstein$^{(3)}$}
\thanks{$^{(3)}$ Zinovy Reichstein was partially supported by
NSERC Discovery and Accelerator Supplement grants}

\subjclass[2000]{20G15}

% 20G05 Representation theory (linear algebraic groups)
% 20G15 Linear algebraic groups over arbitrary fields
% 11E72 (1985-now)  Galois cohomology of linear algebraic groups
% 16K  (1991-now)  Division rings and semisimple Artin rings
% 16K20   (1991-now)  Finite-dimensional
% 20G15 (1973-now)  Linear algebraic groups over arbitrary fields
% 20D15 (1973-now)  Nilpotent groups, $p$-groups
% 20C10  (1973-now)  Integral representations of finite groups

\keywords{Essential dimension, algebraic torus, twisted finite group,
lattice}

\begin{abstract}
The essential dimension is a numerical invariant of an algebraic
group $G$ which may be thought of as a measure of complexity 
of $G$-torsors
 over fields.  A recent theorem of N. Karpenko and A. Merkurjev
 gives a simple formula for the essential dimension of
 a finite $p$-group. We obtain similar formulas
 for the essential $p$-dimension of a broader class
 of groups, which includes all algebraic tori.
\end{abstract}

\maketitle
\tableofcontents

\section{Introduction}
\label{sec:introduction} Throughout this paper $p$ will denote a
prime integer, $k$ a base field of characteristic $\ne p$ and $G$ 
a (not necessarily smooth)
algebraic group defined over $k$. Unless otherwise specified, all
fields are assumed to contain $k$ and all morphisms between them are
assumed to be $k$-homomorphisms.

We begin by recalling the notion of essential dimension
of a functor from~\cite{BF}. Let $\Fields_k$ be the category of
field extensions $K/k$, $\Sets$ be the category of sets,
and $F\colon\Fields_k \to \Sets$ be a covariant functor. 
As usual, given a field extension $k \subset
K_0 \subset K$, we will denote the image of
$\alpha \in F(K)$ under the natural map $F(K) \to F(L)$ by $\alpha_L$.

An object $\alpha \in F(K)$ is said to
\emph{descend} to an intermediate field $k \subseteq K_0 \subseteq K$ 
if $\alpha$ is in the image of the induced map $F(K_0) \to F(K)$.
The \emph{essential dimension} $\ed_k(\alpha)$ 
is defined as the minimum of the transcendence degrees $\trdeg_{k}(K)$ 
taken over all fields $k \subseteq K_0 \subseteq K$ such 
that $\alpha$ descends to $K_0$.  The
essential dimension $\ed_k(F)$ of the functor $F$ is defined as the maximal
value of $\ed_k(\alpha)$, where the maximum is taken over all fields
$K/k$ and all $\alpha \in F(K)$.

Of particular interest to us will be the Galois cohomology functor
$F_G \colonequals H^1(*, G)$, which associates 
to every $K/k$ the set of isomorphism classes of $G$-torsors over $\Spec(K)$.
The essential dimension of this functor is usually
called the {\em essential dimension of} $G$ and is denoted by
the symbol $\ed_k(G)$.  Informally speaking, this number may be thought of
a measure of complexity of $G$-torsors over fields. For example,
if $k$ is an algebraically closed field of characteristic $0$
then groups $G$ of essential dimension $0$ are precisely
the so-called {\em special groups}, i.e., algebraic groups $G/k$ with
the property that every $G$-torsor over $\Spec(K)$ is split,
for every field $K/k$. These groups were classified
by A. Grothendieck~\cite{grothendieck}.

For many groups the essential dimension is hard to compute, even
over the field $\mathbb C$ of complex numbers. The following related
notion is often more accessible. 
Let $F\colon\Fields_k \to \Sets$ be
a covariant functor and $p$ be a prime integer, as above.
The {\em essential $p$-dimension} of
$\alpha \in F(K)$, denoted $\ed_k(\alpha;p)$, 
is defined as the minimal value of
$\ed_k(\alpha_{K'})$, where $K'$ ranges over all finite field
extensions of $K$ whose degree is prime to $p$. 
The essential $p$-dimension of $F$, $\ed_k(F;
p)$ of $F$ is once again, defined as the maximal value of
$\ed_k(\alpha; p)$, where the maximum is taken over all fields $K/k$
and all $\alpha \in F(K)$, and once again we will write
$\ed_k(G; p)$ in place of $\ed_k(F_G; p)$, where
$F_G \colonequals H^1(\ast, G)$ is the Galois cohomology functor.

Note that $\ed_k(\alpha)$, $\ed_k(F)$, $\ed_k(G)$, $\ed_k(\alpha; p)$, etc.,
depend on $k$.
% It is easy to see that
% \begin{equation} \label{field-dependence}
% \text{$\ed_k(G) \ge \ed_{k'}(G_{k'})$ and $\ed_k(G) \ge \ed_{k'}(G_{k'}; p)$}
% \end{equation}
% for any $k \subset k'$; cf. ???
% Here, as usual, $G_{k'}$ denotes the algebraic
% group $G \otimes_{\Spec{k}} \Spec(k')$, defined over $k'$.
We will write $\ed$ instead of $\ed_k$ if the
reference to $k$ is clear from the context.
For background material on essential dimension
we refer the reader to~\cite{BR, Re, RY, BF, Me1}.

We also remark that in the case of the Galois cohomology functor
$F_G$, the maximal value of $\ed_k(\alpha)$ and $\ed_k(\alpha; p)$
in the above definitions is attained in the case where $\alpha$ is
a versal $G$-torsor in the sense of~\cite[Section I.5]{gms}.
Since every generically free linear representation $\rho \colon G \to \GL(V)$
gives rise to a versal $G$-torsor (see~\cite[Example I.5.4]{gms}), we
obtain the inequality
\begin{equation} \label{e.upper-bound}
\ed_k(G;p) \leq \ed_k(G)\leq \dim(V) - \dim(G) \, ;
\end{equation}
see~\cite[Therem 3.4]{Re} or~\cite[Lemma 4.11]{BF}.
(Recall that $\rho$ is called {\em generically free} if there exists
a $G$-invariant dense open subset $U \subset V$ such that
the scheme-theoretic stabilizer of every point of $U$ is trivial.)

N. Karpenko and A. Merkurjev~\cite{KM} recently showed that
the inequality~\eqref{e.upper-bound} is in fact sharp
for finite constant $p$-groups.

\begin{thm} \label{thm.km} Let $G$ be a constant $p$-group and
$k$ be a field containing a primitive $p$th root of unity.  Then
\[ \ed_k(G; p) = \ed_k(G) = \min \, \dim(V) \, , \]
where the minimum is taken over all faithful
$k$-representations $G \hookrightarrow \GL(V)$.
\end{thm}

% In this paper we will view finite $p$-groups $G$ as algebraic
% groups over $k$, and will not assume they are constant, 
% which is to say, the absolute Galois group of $k$ may 
% act non-trivially on the separable points of $G$. 

The goal of this paper is to prove similar formulas
for a broader class of groups $G$. To state our first result, 
let
\begin{equation} \label{e.exact-sequence}
1 \to C \to G \to Q \to 1
\end{equation}
be an exact sequence of algebraic groups over $k$ such that
$C$ is central in $G$ and isomorphic to $\mu_p^r$ for
some $r \ge 0$.
Given a character $\chi \colon C \to \mu_p$, we will, following \cite{KM}, 
denote by $\Rep^{\chi}$ the set of irreducible 
representations $\phi \colon G \to \GL(V)$, defined 
over $k$, such that $\phi(c) = \chi(c) \Id_V$ for every $c \in C$.

\begin{thm} \label{thm:lowerbound}
Assume that $k$ is a field of characteristic $\ne p$
containing a primitive $p$th root of unity. 
Suppose a sequence of $k$-groups of the
form~\eqref{e.exact-sequence} satisfies the following condition:
\[ \gcd \{ \dim(\phi) \, | \, \phi \in \Rep^{\chi} \}
= \min \{ \dim(\phi) \, | \, \phi \in \Rep^{\chi} \} \]
for every character $\chi \colon C \to \mu_p$. (Here, as usual,
$\gcd$ stands for the greatest common divisor.) Then
\[ \ed_k(G;p) \geq \min \dim(\rho) - \dim G \, , \]
where the minimum is taken over
all finite-dimensional $k$-representations $\rho$ of $G$
such that $\rho_{| \, C}$ is faithful.
\end{thm}

Of particular interest to us will be extensions of finite 
$p$-groups by algebraic tori, i.e., $k$-groups $G$ which
fit into an exact sequence of the form
\begin{equation} \label{e.basic-sequence}
 1 \to T \to G \to F \to 1 \, , 
\end{equation}
where $F$ is a finite $p$-group and $T$ is a torus over $k$.
Note that in this paper we will view finite groups $F$ 
as algebraic groups over $k$, and will not assume they are constant, 
which is to say, the absolute Galois group of $k$ may 
act non-trivially on the separable points of $G$. 
For the sake of computing $\ed_k(G; p)$ we may assume that $k$ is a
$p$-closed field (as in Definition \ref{def.p-closure}); see
Lemma~\ref{lem:edPrimeToPClosure}. 
In this situation we will show that

\smallskip
(i) there is a natural choice of a split central subgroup $C \subset G$
in the sequence~\eqref{e.exact-sequence} such that

\smallskip
(ii) the conditions of Theorem~\ref{thm:lowerbound} are always
satisfied. 

\smallskip
(iii) Moreover, if $G$ is isomorphic to the direct product
of a torus and a finite twisted $p$-group, then
a variant of~\eqref{e.upper-bound} yields
an upper bound, matching the lower bound of
Theorem~\ref{thm:lowerbound}.

\smallskip
This brings us to the main result of this paper.
We will say that a representation $\rho\colon G \to \GL(V)$
of an algebraic group $G$ is $p$-faithful if its kernel
is finite and of order prime to $p$.

\begin{thm}
\label{thm:MainTheorem1} Let $G$ be an extension of a (twisted)
finite $p$-group $F$ by an algebraic torus $T$
defined over a field $k$ (of characteristic not $p$).
In other words, we have an exact sequence
\[ 1 \to T \to G \to F \to 1 \, . \]
Denote a $p$-closure of $k$ by
$\pcl{k}$  (see Definition~\ref{def.p-closure}).
% \begin{enumerate}[label=(\alph*), ref=(\alph*)]
% \item 
Then 

\smallskip
(a) $\ed_k(G;p) \ge \min \dim (\rho) - \dim G$,
where the minimum is taken over all $p$-faithful linear
representations $\rho$ of $G_{\pcl{k}}$ over $\pcl{k}$.

\smallskip
\noindent
Now assume that $G$ is the direct product of $T$ and $F$.
Then 

\smallskip
(b) equality holds in (a), and

\smallskip
(c) over $k^{(p)}$ the absolute essential dimension of $G$
and the essential $p$-dimension coincide:
\[ \ed_{k^{(p)}}(G_{k^{(p)}})=\ed_{k^{(p)}}(G_{k^{(p)}};p)=\ed_k(G;p) .\]
% \end{enumerate}
\end{thm}

If $G$ is a $p$-group, a representation $\rho$ is
$p$-faithful if and only if it is faithful. However, for
an algebraic torus, ``$p$-faithful" cannot be replaced by
``faithful"; see~Remark~\ref{rem.p-faithful}.

Theorem~\ref{thm:MainTheorem1} appears to be new
even in the case where $G$ is a twisted cyclic $p$-group,
where it extends earlier work of Rost~\cite{Ro}, Bayarmagnai~\cite{Ba} and
Florence~\cite{Fl}; see Corollary~\ref{cor.cyclic} 
and Remark~\ref{rem.cyclic}.

% 
% We will explore a number of consequences of 
% Theorem~\ref{thm:MainTheorem1}; some of the highlights are listed below.

\smallskip 
If $G$ a direct product of a torus and an abelian 
$p$-group, the value of $\ed_k(G; p)$ given by Theorem~\ref{thm:MainTheorem1}
can be rewritten in terms of the character module $X(G)$; see
Corollary~\ref{cor.lattice}. In particular, we obtain the following 
formula for the essential dimension of a torus.

\begin{thm} \label{thm:MainTheorem2} 
Let $T$ be an algebraic torus defined over a $p$-closed
field $k = \pcl{k}$ of characteristic
$\ne p$. Suppose $\Gamma = \Gal(\sep{k}/k)$ acts 
on the character lattice $X(T)$ via a finite quotient 
$\overline{\Gamma}$. Then
\[ \ed_k(T) = \ed_k(T; p) = \min \rank(L) \, , \]
where the minimum is taken over all exact sequences
of $\ZZ_{(p)}\overline{\Gamma}$-lattices of the form
\[ (0) \to L \to P \to X(T)_{(p)} \to (0)  \, , \] 
where $P$ is permutation and  
$X(T)_{(p)}$ stands for $X(T) \otimes_{\ZZ} \ZZ_{(p)}$.
\end{thm}

In many cases Theorem~\ref{thm:MainTheorem2} renders the value of 
$\ed_k(T)$ computable by known representation-theoretic methods,
e.g., from~\cite{CR}. We will give several examples of such computations 
in Sections~\ref{sect.ed<=1} and~\ref{sect.splitsquare}.
Another application was recently given by Merkurjev (unpublished),
who used Theorem~\ref{thm:MainTheorem2}, in combination 
with techniques from~\cite{Me2}, to show that 
\[ \ed_k(\PGL_{p^r}; p) \ge (r-1)p^r + 1 \]
for any $r \ge 1$.  
(For $r = 2$ the above inequality is the main result of~\cite{Me2}.) 
This represents dramatic improvement over the best previously known 
lower bounds on $\ed_k(\PGL_{p^r})$.  The question of 
computing $\ed_k(\PGL_{p^r})$ is a long-standing open problem; 
for an overview, see~\cite{mr1, mr2}.

%%%%%%%%%%%%%%%%%%%
% $\Gamma = \Gal(k)$ be the absolute Galois group
%  $\ed_k(G; p)$
% as the minimal value of $\rank(L) - \dim(G)$, where
% the minimum is taken over all
% permutation $\ZZ \Gal(k)$-lattices $L$ which
% admit a map $f \colon L \to X(G)$ such that
%  $[X(G) : \Ima(f)]$ is finite and prime to $p$; see Corollary ???.
%%%%%%%%%%%%%%%%%%%

% \smallskip
% (c) We will obtain a formula for the essential $p$-dimension of any
% finite group $G$. The idea is to pass to a Sylow $p$-subgroup of $G$
% defined over $k$ and then apply Theorem~\ref{thm:MainTheorem1};
% see Corollory \ref{cor:Sylow}. In the constant case
% this was done in~\cite[Theorem 4.1, Remark 4.8]{KM}.
% 
% \smallskip
% (d) If $G_1$ and $G_2$ are direct products of tori and $p$-groups, 
% then we will show that
% \[ \ed_k(G_1 \times G_2; p) = \ed_k(G_1; p) + \ed_k(G_2; p) \, ; \]
% see Theorem~\ref{thm.additive}. 
% In the case of constant $p$-groups, this is due to
% Karpenko and Merkurjev~\cite[Theorem 5.1]{KM}.
% 
% \smallskip
It is natural to try to extend the formula 
of Theorem~\ref{thm:MainTheorem1}(b)
to all $k$-groups $G$, whose connected component $G^0$ is a torus.
For example, the normalizer of a maximal torus in 
any reductive $k$-group is of this form.  For the purpose 
of computing $\ed_k(G; p)$ we may assume that $k$ is 
$p$-closed and $G/G^0$ is a $p$-group; in other words,
$G$ is as in Theorem~\ref{thm:MainTheorem1}(a).  Then
\begin{equation} \label{e.conjecture}
\min \dim \mu - \dim(G) \le \ed(G;p) \le \min \dim \rho - \dim G \, ,  
\end{equation}
where the two minima are taken over all $p$-faithful 
representations $\mu$, and
$p$-generically free representations $\rho$,
respectively. Here we say that a representation 
$\rho$ of $G$ is $p$\textit{-generically free} if
the $\ker(\rho)$ is finite of order prime to $p$, and $\rho$ descends
to a generically free representation of $G/\ker(\rho)$. The 
upper bound in~\eqref{e.conjecture}
follows from \eqref{e.upper-bound}, in combination with
Theorem~\ref{thm:edQuotient}; the lower bound 
is Theorem~\ref{thm:MainTheorem1}(a).
If $G$ is a direct product of a torus and a $p$-group, 
then every $p$-generically free representation 
is $p$-faithful (see Lemma~\ref{lem.rigid}). In this case 
the lower and upper bounds of~\eqref{e.conjecture} 
coincide, yielding the exact value of $\ed_k(G; p)$
of Theorem~\ref{thm:MainTheorem1}(b).
However, if we only assume $G$ is a $p$-group 
extended by a torus, then faithful
$G$-representations no longer need to be generically free. 
We do not know how to bridge the gap between the upper 
and the lower bound in~\eqref{e.conjecture} in this generality;
however, in all of the specific
examples we have considered, the upper bound turned out to be sharp.
We thus put forward the following conjecture.

\begin{conj} \label{conjecture}
Let $G$ be an extension of a $p$-group by a torus, defined over
a field $k$ of characteristic $\ne p$. 
Then \[ \ed(G;p)=\min \dim \rho - \dim G, \] 
where the minimum is taken over all $p$-generically free
representations $\rho$ of $G_{\pcl{k}}$ over $\pcl{k}$.
\end{conj}

The rest of the paper is structured as follows.
Theorem~\ref{thm:lowerbound} is proved in
Section~\ref{sect.lowerbound}. Section~\ref{sec:primeToPClosure} is
devoted to preliminary material on the $p$-closure of a field.  
Theorem~\ref{thm:MainTheorem1}(a) is proved in Sections~\ref{sect.C(G)}
and~\ref{sect.MainTheorem1a}. 
In Section~\ref{sect.isogeny} we will show that if $G \to Q$ is a
$p$-isogeny then $\ed_k(G; p) = \ed_k(Q; p)$. This result
playes a key role in the proof of Theorem~\ref{thm:MainTheorem1}(b) 
in Section~\ref{sect.MainTheorem1b}. At the end of  
Section~\ref{sect.MainTheorem1b} we prove
a formula for the essential $p$-dimension of any finite 
group $G$ by passing to a Sylow $p$-subgroup defined over $k$;
see Corollory \ref{cor:Sylow}. In Section~\ref{sect.additivity} 
we prove the following Additivity 
Theorem~\ref{thm.additive}:
If $G_1$ and $G_2$ are direct products of tori and $p$-groups, then
\[ \ed_k(G_1 \times G_2; p) = \ed_k(G_1; p) + \ed_k(G_2; p) \, . \]
In Section~\ref{sect.modules+lattices}
we restate and amplify Theorem~\ref{thm:MainTheorem1}(b) (with $G$ abelian)
in terms of $\Gal(\sep{k}/k)$-modules; in particular, 
Theorem~\ref{thm:MainTheorem2} stated above is a special case 
of Corollary~\ref{cor.lattice} which is proved there. 
In Section~\ref{sect.MainTheorem1c} we prove 
Theorem~\ref{thm:MainTheorem1}(c) by 
using Theorem~\ref{thm:MainTheorem1}(b), 
additivity, and the lattice perspective 
from Section~\ref{sect.modules+lattices}. 
The last two sections are intended to illustrate our 
results by computing essential dimensions of specific 
algebraic tori.  In Section~\ref{sect.ed<=1} we classify 
algebraic tori $T$ of essential $p$-dimension 
$0$ and $1$; see Theorems~\ref{thm.ed=0} and
\ref{thm.ed=1}.  In Section~\ref{sect.splitsquare} we compute 
the essential $p$-dimension of all tori $T$ over a $p$-closed
field $k$, which are split by a cyclic 
extension $l/k$ of degree dividing $p^2$.

\section{Proof of Theorem~\ref{thm:lowerbound}}
\label{sect.lowerbound}

% Let $\zeta$ be a $p$-th root of unity in $\sep{k}$. Set $k' = k(\zeta)$.
% Since $[k': k]$ is prime to $p$, $\ed(G_{k'}; p) = \ed(G; p)$; see, e.g.,
% \cite[Proposition 1.5]{Me1}.

Denote by $C^\ast\colonequals\Hom(C,\mu_p)$ the character group of
$C$. Let $E \to \Spec K$ be a versal $Q$-torsor \cite[Example
5.4]{gms}, where $K/k$ is some field extension, and let $\beta
\colon C^\ast \to \Br_p(K)$ denote the homomorphism that sends $\chi
\in C^\ast$ to the image of $E\in H^1(K,Q)$ in $\Br_p(K)$ under the
map
\[ H^1(K,Q)\to H^2(K,C)\stackrel{\chi_\ast}\to H^2(K,\mu_p) = \Br_p(K) \]
given by composing the connecting map with $\chi_{\ast}$. Then
there exists a basis $\chi_1,\dotsc,\chi_r$ of $C^\ast$  such that
\begin{equation}
\label{eq:lowerBound}
\ed_k(G;p) \geq \sum_{i=1}^r \ind \beta(\chi_i) - \dim G,
\end{equation}
 see \cite[Theorem 4.8, Example 3.7]{Me1}.
Moreover, by \cite[Theorem 4.4, Remark 4.5]{KM}
\[ \ind \beta(\chi_i)= \gcd \dim (\rho) \, ,  \]
where the greatest common divisor is taken
over all (finite-dimensional) representations $\rho $ of $G$
such that $\rho_{| \, C}$ is scalar multiplication
by $\chi_i$.
By our assumption, $\gcd$ can be replaced by $\min$. Hence,
for each $i \in \{1,\dotsc,r\}$ we can choose
a representation $\rho_i$ of $G$ with
\[\ind \beta(\chi_i) = \dim (\rho_i)\]
such that $(\rho_i)_{| \, C}$ is scalar multiplication by $\chi_i$.

Set $\rho\colonequals \rho_1\oplus \dotsb \oplus \rho_r$.
The inequality~\eqref{eq:lowerBound} can be written as
\begin{equation}
\label{eq:lowerBoundRep}
\ed_k(G;p)\geq \dim(\rho)- \dim G.
\end{equation}
Since $\chi_1,\dotsc,\chi_r$ forms a basis of $C^\ast$
the restriction of $\rho$ to $C$ is faithful. This proves the theorem.
\qed

\section{The $p$-closure of a field}
\label{sec:primeToPClosure}

Let $K$ be a field extension of $k$ and $\alg{K}$ an algebraic closure.
We will construct a field $\pcl{K}/K$ in $\alg{K}$ with all finite subextensions of $\pcl{K}/K$ of degree prime to $p$ and all finite subextensions of $\alg{K}/\pcl{K}$ of degree a power of $p$.

Fix a separable closure $\sep{K}\subset \alg{K}$ of $K$ and denote
$\Gamma=\Gal(\sep{K}/K)$. Recall that $\Gamma$ is profinite and has
Sylow-$p$ subgroups which enjoy similar properties as in the finite
case, see for example \cite{RZ} or \cite{Wi}. Let $\Phi$ be a
Sylow-$p$ subgroup of $\Gamma$ and $\sep{K}^{\Phi}$ its fixed field.
% (it is a maximal separable extension of $K$ such that
%all finite subextensions are of degree prime to $p$).

\begin{defn} \label{def.p-closure}
We call the field
\[
\pcl{K}=\{a\in \alg{K}|a \mbox{ is purely inseparable over
}\sep{K}^{\Phi}\}
\]
a {\em $p$-closure} of $K$. A field
$K$ will be called {\em $p$-closed} if $K$=$\pcl{K}$.
\end{defn}
Note that $\pcl{K}$ is unique in $\alg{K}$ only up to the choice of
a Sylow-$p$ subgroup $\Phi$ in $\Gamma$. The notion of being $p$-closed
does not depend on this choice. 

% We record some immediate properties of the $p$-closure.

\begin{prop}\label{ptopprop}\hfill
\begin{enumerate}[label=(\alph*), ref=(\alph*)]
\item\label{ptop1} $\pcl{K}$ is a direct limit of finite
extensions $K_i/K$ of degree prime to $p$.
%; in particular, $[\pcl{K}:K]$ is prime to $p$.
\item\label{ptop2} Every finite extension of $\pcl{K}$
is separable of degree a power of $p$; in particular,
$\pcl{K}$ is perfect.
\item\label{ptop3} The cohomological dimension
of $\Psi=\Gal(\alg{K}/\pcl{K})$ is ${\rm cd}_q(\Psi)=0$ for
any prime $q\ne p$.
\end{enumerate}
\end{prop}
\begin{proof}
\ref{ptop1} First note that $\sep{K}$ is the limit of the directed
set $\{\sep{K}^N\}$ over all normal subgroups $N \subset \Gamma$ of
finite index. Let
\[
\mathcal{L}=\{\sep{K}^{N \Phi}| N\mbox{ normal with finite index in
$\Gamma$}\}.
\]
This is a directed set, and since $\Phi$ is Sylow, the index of $N
\Phi$ in $\Gamma$ is prime to $p$. Therefore $\mathcal{L}$ consists
of finite separable extensions of $K$ of degree prime to $p$.
Moreover, $\sep{K}^{\Phi}$ is the direct limit of fields $L$ in
$\mathcal{L}$.

If $\Char k=0$, $\pcl{K}=\sep{K}^{\Phi}$ and we are done. Otherwise
suppose $\Char k=q\ne p$. Let
\[
\mathcal{E}=\{E\subset \alg{K}| E/L\mbox{ finite and
purely inseparable for some } L\in\mathcal{L}\}.
\]
$\mathcal{E}$ consists of finite extensions of $K$ of degree prime
to $p$, because a purely inseparable extension has degree a power of
$q$. One can check that $\mathcal{E}$ forms a directed set.

Finally note that if $a$ is purely inseparable over $\sep{K}^{\Phi}$
with minimal polynomial $x^{q^{n}}-l$ (so that $l\in
\sep{K}^{\Phi}$), then $l$ is already in some $L\in \mathcal{L}$
since $\sep{K}^{\Phi}$ is the limit of $\mathcal{L}$. Thus $a\in
E=L(a)$ which is in $\mathcal{E}$ and we conclude that $\pcl{K}$ is
the direct limit of $\mathcal{E}$.

\ref{ptop2} $\pcl{K}$ is the purely inseparable closure of
$\sep{K}^{\Phi}$ in $\alg{K}$ and $\alg{K}/\pcl{K}$ is separable,
see \cite[2.2.20]{Win}. Moreover, $\Gal(\alg{K}/\pcl{K})\simeq
\Gal(\sep{K}/\sep{K}^{\Phi})=\Phi$ is a pro-$p$ group and so every
finite extension of $\pcl{K}$ is separable of degree a power of $p$.

\ref{ptop3} See \cite[Cor. 2, I. 3]{serre-gc}.
\end{proof}
%\begin{rem}
%Given an arbitrary algebraic extension $L/K$ there may not exist
%an intermediate subfield $K \subset \pcl{K} \subset L$ such that
%the degree of
%every finite subextension of $\pcl{K}/K$ is prime to $p$ and
%the degree of every finite subextension
%of $L/\pcl{K}$ is a power of $p$. An intermediate
%subfield with these properties exists
%if and only if $L=K_sK_{is}$,
% where $K_s$ (respectively, $K_{is})$ denotes
%the separable (purely inseparable) closure of $K$ in $L$.
%\end{rem}
We call a covariant functor $\mathcal{F} \colon 
\Fields/k\to \Sets$ {\em limit-preserving}
if for any directed system of fields $\{K_i\}$,
$\displaystyle{\mathcal{F}(\lim_{\rightarrow}K_i)=
\lim_{\rightarrow}\mathcal{F}(K_i)}$.
For example if $G$ is an algebraic group, the Galois
cohomology functor $H^1(*,G)$ is limit-preserving; see \cite[2.1]{Ma}. 

\begin{lem}
\label{lem:edPrimeToPClosure}
Let $\mathcal{F}$ be limit-preserving and $\alpha \in \mathcal{F}(K)$ an object.
Denote the image of $\alpha$ in $\mathcal{F}(\pcl{K})$ by $\alpha_{\pcl{K}}$.
\begin{enumerate}[label=(\alph*), ref=(\alph*)]
\item\label{edptop1} $\ed_k(\alpha;p)=\ed_k(\alpha_{\pcl{K}};p)=
\ed_k(\alpha_{\pcl{K}})$.
\item\label{edptop2} $\ed_k(\mathcal{F};p)=\ed_{\pcl{k}}(\mathcal{F};p)$.
\end{enumerate}
\end{lem}

\begin{proof}
\ref{edptop1}
The inequalities $\ed(\alpha;p)\ge\ed(\alpha_{\pcl{K}};p)=\ed(\alpha_{\pcl{K}})$ are clear from the definition and Proposition~\ref{ptopprop}(b) since $\pcl{K}$ has no finite extensions of degree prime to $p$.
It remains to prove $\ed(\alpha;p)\le\ed(\alpha_{\pcl{K}})$.
If $L/K$ is finite of degree prime to $p$,
\begin{equation}\label{eq.ptp}
\ed(\alpha;p)=\ed(\alpha_L;p),
\end{equation}
cf. \cite[Proposition 1.5]{Me1} and its proof.
For the $p$-closure $\pcl{K}$ this 
is similar and uses \eqref{eq.ptp} repeatedly:

Suppose there is a subfield $K_0\subset \pcl{K}$ and
$\alpha_{\pcl{K}}$ comes from an element $\beta\in
\mathcal{F}(K_0)$, so that $\beta_{\pcl{K}}=\alpha_{\pcl{K}}$. Write
$\pcl{K}=\lim\mathcal{L}$, where $\mathcal{L}$ is a direct system of
finite prime to $p$ extensions of $K$. Then $K_0=\lim\mathcal{L}_0$
with $\mathcal{L}_0=\{L\cap K_0|L\in\mathcal{L}\}$ and by assumption
on $\mathcal{F}$,
$\displaystyle{\mathcal{F}(K_0)=\lim_{L^\prime\in\mathcal{L}_0}\mathcal{F}(L^\prime)}$.
Thus there is a field $L^\prime=L\cap K_0$ ($L\in \mathcal{L}$) and
$\gamma\in\mathcal{F}(L^\prime)$ such that $\gamma_{K_0}=\beta$.
Since $\alpha_L$ and $\gamma_L$ become equal over $\pcl{K}$, after
possibly passing to a finite extension, we may assume they are equal
over $L$ which is finite of degree prime to $p$ over $K$. Combining
these constructions with \eqref{eq.ptp} we see that
\[
\ed(\alpha;p)=\ed(\alpha_L;p)=\ed(\gamma_L;p)\le\ed(\gamma_L)\le\ed(\alpha_{\pcl{K}}).
\]

\ref{edptop2} This follows immediately from \ref{edptop1}, taking
$\alpha$ of maximal essential $p$-dimension.
\end{proof}

\begin{prop}
\label{prop:bijectionPrimeToPClosure} Let
$\mathcal{F},\mathcal{G}\colon \Fields/k \to \Sets$ be
limit-preserving functors and $\mathcal{F} \to \mathcal{G}$ a
natural transformation. If the map
\[\mathcal{F}(K) \to \mathcal{G}(K)\]
is bijective (resp.~surjective) for any $p$-closed field extension $K/k$
then
\[\ed(\mathcal{F};p)= \ed(\mathcal{G};p) \quad (\mbox{resp. }\ed(\mathcal{F};p)\geq\ed(\mathcal{G};p)).\]
\end{prop}
\begin{proof}

Assume the maps are surjective. By Proposition~\ref{ptopprop}\ref{ptop1},
the natural transformation is $p$-surjective, in the terminology of
\cite{Me1}, so we can apply \cite[Prop.\ 1.5]{Me1} to conclude
$\ed(\mathcal{F};p)\geq\ed(\mathcal{G};p)$.

Now assume the maps are bijective. Let $\alpha$ be in
$\mathcal{F}(K)$ for some $K/k$ and $\beta$ its image in
$\mathcal{G}(K)$. We claim that
%\begin{equation}\label{eqel}
$\ed(\alpha; p)=\ed(\beta; p)$.
%\end{equation}
First, by Lemma~\ref{lem:edPrimeToPClosure} we can assume that $K$
is $p$-closed and it is enough to prove that
$\ed(\alpha)=\ed(\beta)$.

Assume that $\beta$ comes from $\beta_0\in \mathcal{G}(K_0)$ for
some field $K_0\subset K$.
Any finite prime to $p$
extension of $K_0$ is isomorphic to a subfield 
of $K$ (cf.~\cite[Lemma 6.1]{Me1})
and so also any $p$-closure of $K_0$ (which has the
same transcendence degree over $k$).
We may therefore assume that $K_0$ is $p$-closed.
By assumption $\mathcal{F}(K_0)\rightarrow
\mathcal{G}(K_0)$ and $\mathcal{F}(K)\rightarrow \mathcal{G}(K)$ are
bijective. The unique element $\alpha_0\in\mathcal{F}(K_0)$ which
maps to $\beta_0$ must therefore map to $\alpha$ under the natural
restriction map. This shows that $\ed(\alpha)\le\ed(\beta)$. The
other inequality always holds and the claim follows.

Taking $\alpha$ maximal with respect to its essential dimension, we
obtain
$\ed(\mathcal{F};p)=\ed(\alpha;p)=\ed(\beta;p)\le\ed(\mathcal{G};p)$.
\end{proof}

\section{The group $C(G)$}
\label{sect.C(G)}

As we indicated in the Introduction, our proof of
Theorem~\ref{thm:MainTheorem1}(a) will rely on
Theorem~\ref{thm:lowerbound}. To apply 
Theorem~\ref{thm:lowerbound}, we need to construct 
a split central subgroup $C$ of $G$. In this section,
we will explain how to construct this subgroup (we will call it $C(G)$)
and discuss some of its properties.

Recall that an algebraic group $G$ over a field $k$ is said to be
{\em of multiplicative type} if $G_{\sep{k}}$ is diagonalizable over
the separable closure $\sep{k}$ of $k$; cf.,
e.g.,~\cite[Section 3.4]{voskresenskii}.
Here, as usual, $G_{k'} \colonequals
G \times_{\Spec{k}} \Spec(k')$ for any field extension $k'/k$.
Connected groups of multiplicative type are precisely the algebraic tori.

We will use the following common conventions in working with
an algebraic group $A$ of multiplicative type over $k$.

\begin{itemize}

\smallskip
\item
We will denote the character group of $A$ by $X(A)$.

\smallskip
\item
Given a field extension $l/k$, $A$ is split over $l$ if and only if
the absolute Galois group $\Gal(\sep{l}/l)$ acts trivially on
$X(A)$.

\smallskip
\item
We will write $A[p]$ for the $p$-torsion
subgroup $\{ a \in A \, | \, a^p = 1 \}$ of $A$.
Clearly $A[p]$ is defined over $k$.
\end{itemize}

Let $T$ be an algebraic torus.  It is well 
known how to construct a maximal split subtorus 
of $T$, see for example \cite[8.15]{Bo} or \cite[7.4]{Wa}.
The following definition is a variant of this.

\begin{defn}
Let $A$ be an algebraic group of multiplicative type over $k$.
Let $\Delta(A)$ be the $\Gamma$-invariant subgroup of $X(A)$
generated by elements of the form $x - \gamma(x)$, as $x$ ranges
over $X(A)$ and $\gamma$ ranges over $\Gamma$. Define
\[ \Split(A)=\Diag( X(A)/\Delta(A) )\,. \]
\end{defn}
Here $\Diag$ denotes the anti-equivalence between 
continuous $\ZZ\Gamma$-modules and algebraic groups 
of multiplicative type, cf. \cite[7.3]{Wa}. \par
%In the sequel let $\Gamma=\Gal(\sep{k}/k)$ acts on $X(A)$ through some finite quotient $\overline{\Gamma}$. The group
%$\overline{\Gamma}$ is the Galois group of an extension $l/k$ so that
%$A$ splits over $l$.

\begin{defn} \label{def.C(G)}
Let $G$ be an extension of a finite $p$-group by a torus, defined over 
a field $k$, as in~\eqref{e.basic-sequence}. Then
\[ C(G) \colonequals \Split(Z(G)[p]) \, ,  \]
where $Z(G)$ denotes the centre of $G$.
\end{defn}

\begin{lem} \label{lem3.1}
Let $A$ be an algebraic group of multiplicative type over $k$.

\begin{enumerate}[label=(\alph*), ref=(\alph*)]
\item \label{SplitAa} $\Split(A)$ is split over $k$,
\item \label{SplitAb} $\Split(A)
= A$ if and only if $A$ is split over $k$,
\item \label{SplitAc} If
$B$ is a $k$-subgroup of $A$ then $\Split(B) \subset \Split(A)$.
\item \label{SplitAd} For $A=A_1 \times A_2$,
$\Split(A_1 \times A_2) = \Split(A_1) \times \Split(A_2)$,
\item\label{SplitAe} If $A[p] \ne \{ 1 \}$ and $A$ is
split over a Galois extension $l/k$, such that $\overline{\Gamma}=\Gal(l/k)$ is a
$p$-group, then $\Split(A) \ne \{ 1 \}$.
\end{enumerate}
\end{lem}

\begin{proof}
Parts \ref{SplitAa}, \ref{SplitAb},
\ref{SplitAc} and \ref{SplitAd} easily follow from the definition.

Proof of \ref{SplitAe}: By part \ref{SplitAc}, it suffices to show
that $\Split(A[p]) \ne \{ 1 \}$. Hence, we may assume that $A =
A[p]$ or equivalently, that $X(A)$ is a finite-dimensional
$\FF_p$-vector space on which the $p$-group $\overline{\Gamma}$ acts. Any such
action is upper-triangular, relative to some $\FF_p$-basis $e_1,
\dots, e_n$ of $X(A)$; see, e.g.,~\cite[Proposition 26,
p.64]{serre-rep}. That is,
\[ \text{$\gamma(e_i) = e_i + $ ($\FF_p$-linear combination of $e_{i+1},
\ldots, e_n$)} \]
for every $i = 1, \dots, n$ and every $\gamma \in \overline{\Gamma}$.
Our goal is to show that $\Delta(A) \ne X(A)$. Indeed,
every element of the form
$x - \gamma(x)$ is contained in the $\Gamma$-invariant
submodule $\Span(e_2, \dots, e_n)$. Hence, these elements
cannot generate all of $X(A)$.
\end{proof}

\begin{prop} \label{prop.C(G)}
Suppose $G$ is an extension of a $p$-group by a torus, 
defined over a $p$-closed field $k$. Suppose $N$ is a normal 
subgroup of $G$ defined over $k$. Then the following conditions
are equivalent:

\smallskip
(i) $N$ is finite of order prime to $p$,

\smallskip
(ii) $N \cap C(G) = \{ 1 \}$,

\smallskip
(iii) $N \cap Z(G)[p] = \{ 1 \}$,
\end{prop}

In particular, taking $N = G$, we see that $C(G) \neq \{ 1 \}$ if $G\neq\{1\}$.

\begin{proof} (i) $\Longrightarrow$ (ii) is obvious, since $C(G)$ 
is a $p$-group.

(ii) $\Longrightarrow$ (iii). Assume the contrary:
$A \colonequals N \cap Z(G)[p] \neq \{ 1 \}$. 
By Lemma~\ref{lem3.1} 
\[ \{ 1 \} \neq C(A) \subset N \cap C(Z(G)[p])  
= N \cap C(G) \, , \]
contradicting (ii).

Our proof of the implication (iii) $\Longrightarrow$ (i), 
will rely on the following 

\smallskip
{\bf Claim:}
Let $M$ be a non-trivial normal finite $p$-subgroup of $G$ such 
that the commutator $(G^0,M)=\{1\}$. Then $M \cap Z(G)[p] \neq \{1\}$.

\smallskip
To prove the claim, note that $M(\sep{k})$ is non-trivial and
the conjugation action of $G(\sep{k})$ on $M(\sep{k})$ factors through an action 
of the $p$-group $(G/G^0)(\sep{k})$. Thus each orbit has $p^n$ elements for 
some $n \ge 0$; consequently, the number of fixed points is
divisible by $p$. The intersection $(M \cap Z(G))(\sep{k})$ is 
precisely the fixed point set for this action; hence,
$M \cap Z(G)[p] \neq \{ 1 \}$. This proves the claim.

We now continue with the proof of the implication (iii) $\Longrightarrow$ (i). 
For notational convenience, set $T \colonequals G^0$.
Assume that $N\triangleleft G$ and $N \cap Z(G)[p] = \{ 1 \}$.
Applying the claim to the normal subgroup 
$M \colonequals (N \cap T)[p]$ of $G$, we see that 
$(N\cap T)[p] = \{ 1 \}$, i.e., $N \cap T$ is a finite group of order
prime to $p$.  The exact sequence 
\begin{equation} \label{extension2} 
1 \to N\cap T \to N \to \overline{N} \to 1 \, ,  
\end{equation}
where $\overline{N}$ is the image of $N$ in $G/T$, 
shows that $N$ is finite. Now observe that
for every $r \ge 1$, the commutator $(N, T[p^r])$ is a $p$-subgroup
of $N \cap T$. Thus $(N, T[p^r]) = \{ 1 \}$ for every $r \ge 1$.
We claim that this implies $(N, T) = \{ 1 \}$ by Zariski density.
If $N$ is smooth, this is straightforward; 
see~\cite[Proposition 2.4, p. 59]{Bo}. If $N$ is not smooth, 
note that the map $c \colon N \times T \to G$ 
sending $(n, t)$ to the commutator $n t n^{-1} t^{-1}$
descends to $\overline{c} \colon 
\overline{N} \times T \to G$ (indeed, $N \cap T$ clearly commutes with $T$).
Since $|\overline{N}|$ is a power of $p$ 
and $\Char(k) \ne p$, $\overline{N}$ is smooth over $k$, and we can 
pass to the separable closure $\sep{k}$ and apply 
the usual Zariski density argument to show that 
the image of $\overline{c}$ is trivial.

We thus conclude that $N \cap T$ is central in $N$.
Since $\gcd(|N\cap T|, \overline{N})=1$, by
\cite[Corollary 5.4]{Sch} the extension~\eqref{extension2}
splits, i.e., $N \simeq (N \cap T) \times \overline{N}$. This turns 
$\overline{N}$ into a subgroup of $G$ satisfying the conditions 
of the claim. Therefore $\overline{N}$ is trivial 
and $N=N\cap T$ is a finite group of order prime to $p$, 
as claimed.
\end{proof}

For future reference, we record the following obvious
consequence of the equivalence of conditions (i) and (ii) in
Proposition~\ref{prop.C(G)}.

\begin{cor} \label{cor.C(G)}
Let $k = \pcl{k}$ be a $p$-closed field and $G$ be 
an extension of a $p$-group by a torus, defined 
over $k$, as in~\eqref{e.basic-sequence}.
A finite-dimensional representation $\rho$ of $G$ defined over $k$
is $p$-faithful if and only $\rho_{| \, C(G)}$ is faithful.
\qed
\end{cor}

\section{Proof of Theorem~\ref{thm:MainTheorem1}(a)}
\label{sect.MainTheorem1a}

The key step in our proof will be the following proposition.

\begin{prop} \label{prop3.2}
Let $k$ be a $p$-closed field, and $G$ be an extension of
a $p$-group by a torus, as in~\eqref{e.basic-sequence}.
% \[ 1 \to T \to G \to F \to 1 \]
% be an exact sequence of algebraic $k$-groups where
% $T$ is a torus, and $F$ is a (twisted) finite $p$-group.
Then the dimension of every irreducible representation
of $G$ over $k$ is a power of $p$.
\end{prop}

Assuming Proposition~\ref{prop3.2} we can easily complete the proof of
Theorem~\ref{thm:MainTheorem1}(a). Indeed, by 
Proposition~\ref{prop:bijectionPrimeToPClosure}
we may assume that
$k = \pcl{k}$ is $p$-closed. In particular, since we are assuming
that $\Char(k) \ne p$, this implies that $k$ contains 
a primitive $p$th root of unity.
(Indeed, if $\zeta$ is a $p$-th root of unity in $\sep{k}$ then
$d = [k(\zeta): k]$ is prime to $p$; hence, $d = 1$.)
Proposition~\ref{prop3.2} tells us that 
Theorem~\ref{thm:lowerbound} can be applied to
the exact sequence
\begin{equation} \label{e.soc}
1 \to C(G) \to G \to Q \to 1 \, .
\end{equation}
This yields 
\begin{equation} \label{e.lower-bound2}
\ed(G;p) \ge \min \; \dim(\rho) - \dim(G) \, ,
\end{equation}
where the minimum is taken over all representations
$\rho \colon G \to \GL(V)$ such that $\rho_{|C(G)}$ is faithful.
Corollary~\ref{prop.C(G)} now tells us
that $\rho_{|C(G)}$ is faithful if and only if
$\rho$ is $p$-faithful, and Theorem~\ref{thm:MainTheorem1}(a)
follows.
\qed

\smallskip
The rest of this section will be devoted to the proof of
Proposition~\ref{prop3.2}.  We begin by settling it in the case where
$G$ is a finite $p$-group.

\begin{lem} \label{lem3.2a} 
Proposition~\ref{prop3.2} holds if $G$ is a finite $p$-group.
\end{lem}

\begin{proof} Choose a finite Galois field extension $l/k$ such that
(i) $G$ is constant over $l$ and
(ii) every irreducible linear representation
of $G$ over $l$ is absolutely irreducible.
Since $k$ is assumed to be $p$-closed, $[l:k]$ is a power of $p$. 

Let $A\colonequals k[G]^\ast$ be the dual Hopf algebra 
of the coordinate algebra of $G$. 
By \cite[Section 8.6]{Ja} a $G$-module structure on 
a $k$-vector space $V$ is equivalent 
to an $A$-module structure on $V$. Now assume that $V$ is 
an irreducible $A$-module and let
$W\subseteq V\otimes_k l$
be an irreducible $A\otimes_k l$-submodule. 
Then by \cite[Theorem 5.22]{Ka} there exists 
a divisor $e$ of $[l:k]$ such that 
\[
V\otimes l \simeq e\left(\bigoplus_{i=1 }^r {}^{\sigma_i}W\right) 
\, , 
\] 
where $\sigma_i \in \Gal(l/k)$ and $\{{}^{\sigma_i}W \mid 1 \leq i \leq r\}$ 
are the pairwise non-isomorphic Galois conjugates of $W$.
By our assumption on $k$, $e$ and $r$ are powers of $p$ and by our choice
of $l$, $\dim_l W = \dim_l({}^{\sigma_1}W) = \ldots = \dim_l({}^{\sigma_r}W)$ 
is also a power of $p$, since it divides the order of $G_l$.
%because they are monomial, which follows for example from \cite{LP}.  
Hence, so is 
$\dim_k(V) = \dim_l V\otimes l = e(\dim_l {}^{\sigma_1}W+ \dotsb +
\dim_l {}^{\sigma_r}W)$.
\end{proof}

Our proof of Proposition~\ref{prop3.2} in full generality 
will based on leveraging Lemma~\ref{lem3.2a} as follows. 

\begin{lem} \label{lem3.3}
Let $G$ be an algebraic group defined over a field $k$ and 
\[ F_1 \subseteq F_2 \subseteq \dots \subset G \]
be an ascending sequence of finite $k$-subgroups whose union
$\cup_{n \ge 1} F_n$ is Zariski dense in $G$.
If $\rho \colon G \to \GL(V)$ is an irreducible representation
of $G$ defined over $k$ then $\rho_{| \, F_i}$ is irreducible for
sufficiently large integers $i$.
\end{lem}
 
\begin{proof}
For each $d = 1, ..., \dim(V) - 1$ consider the $G$-action 
on the Grassmannian $\Gr(d, V)$ of $d$-dimensional subspaces of $V$. 
Let $X^{(d)} = \Gr(d, V)^G$ and
$X_i^{(d)} = \Gr(d, V)^{F_i}$
be the subvariety of $d$-dimensional $G$- (resp.~$F_i$-)invariant subspaces of $V$.
Then $X_1^{(d)} \supseteq X_2^{(d)} \supseteq \ldots$ and
since the union of the groups $F_i$ is dense in $G$, 
\[ X^{(d)} = \cap_{i \ge 0} X_i^{(d)} \, . \]
By the Noetherian property of $\Gr(d, V)$, we have
$X^{(d)} = X_{m_d}^{(d)}$ for some $m_d \ge 0$.

Since $V$ does not have any $G$-invariant $d$-dimensional 
$k$-subspaces, we know that $X^{(d)}(k) = \emptyset$.
Thus, $X_{m_d}^{(d)} (k) = \emptyset$, i.e., 
$V$ does not have any $F_{m_d}$-invariant $d$-dimensional 
$k$-subspaces.
Setting $m \colonequals \max \{ m_1, \dots, m_{\dim(V) - 1} \}$, we see that
$\rho_{| \, F_m}$ is irreducible.
\end{proof}

We now proceed with the proof of Proposition~\ref{prop3.2}.
By Lemmas~\ref{lem3.2a} and~\ref{lem3.3}, it suffices to 
construct a sequence of finite $p$-subgroups
\[ F_1 \subseteq F_2 \subseteq \dots \subset G \]
defined over $k$ whose union $\cup_{n \ge 1} F_n$ is Zariski 
dense in $G$. 

In fact, it suffices to construct one $p$-subgroup 
$F' \subset G$, defined over $k$ such that $F'$
surjects onto $F$. Indeed, once $F'$ is constructed,
we can define $F_i \subset G$ as 
the subgroup generated  by $F'$ and $T[p^i]$, for every $i \ge 0$.
Since $\cup_{n \ge 1} F_n$ contains both $F'$ and 
$T[p^i]$, for every $i \ge 0$ it is Zariski dense 
in $G$, as desired.

The following lemma, which establishes the existence of $F'$, 
is thus the final step in our proof of Proposition~\ref{prop3.2}
(and hence, of Theorem~\ref{thm:MainTheorem1}(a)).

\begin{lem} \label{lem.schneider}
Let $1 \to T \to G \xrightarrow{\pi} F \to 1$ be an extension 
of a $p$-group $F$ by a torus $T$ over $k$. 
Then $G$ has a finite $p$-subgroup $F'$ with $\pi(F')= F$. 
\end{lem}

In the case where $F$ is split and $k$ is algebraically closed
this is proved in~\cite[p. 564]{cgr}; cf. also the proof of
\cite[Lemme 5.11]{bs}. 

\begin{proof}
Denote by $\Ext^1(F,T)$ the group of equivalence 
classes of extensions of $F$ by $T$.
We claim that $\Ext^1(F,T)$ is torsion.
Let $\Ex^1(F,T)\subset\Ext^1(F,T)$ be the classes 
of extensions which have a scheme-theoretic 
section (i.e. $G(K)\to F(K)$ is surjective for all $K/k$).
There is a natural isomorphism $\Ex^1(F,T)\simeq H^2(F,T)$, where 
the latter one denotes Hochschild cohomology, 
see~\cite[III. 6.2, Proposition]{DG}.
By~\cite{Sc2} the usual restriction-corestriction 
arguments can be applied in Hochschild cohomology 
and in particular, $m\cdot H^2(F,T)=0$ where $m$ is the order of $F$.
Now recall that $M\mapsto\Ext^i(F,M)$ and $M\mapsto\Ex^i(F,M)$ 
are both derived functors of the crossed homomorphisms 
$M\mapsto\Ex^0(F,M)$, where in the first case $M$ 
is in the category of $F$-module sheaves and in 
the second, $F$-module functors, cf. \cite[III. 6.2]{DG}.
Since $F$ is finite and $T$ an affine scheme, 
by \cite[Satz 1.2 \& Satz 3.3]{Sc1} there is an exact sequence 
of $F$-module schemes $1\to T \to M_1 \to M_2 \to 1$ 
and an exact sequence
$\Ex^0(F,M_1)\to\Ex^0(F,M_2)\to\Ext^1(F,T)\to H^2(F,M_1)\simeq\Ex^1(F,M_1)$.
The $F$-module sequence also induces a long exact 
sequence on $\Ex(F,*)$ and we have a diagram
\[
\xymatrix{
&&\Ext^1(F,T)\ar@{->}[dr]&\\
\Ex^0(F,M_1)\ar@{->}[r]&\Ex^0(F,M_2)\ar@{->}[ur]\ar@{->}[dr]&&\Ex^1(F,M_1)\\
&&\Ex^1(F,T)\ar@{->}[ur]\ar@{^{(}->}[uu]&
}
\]
An element in $\Ext^1(F,T)$ can thus be killed first 
in $\Ex^1(F,M_1)$ so it comes from $\Ex^0(F,M_2)$. 
Then kill its image in $\Ex^1(F,T)\simeq H^2(F,T)$, so it comes from $\Ex^0(F,M_1)$, hence is $0$ in $\Ext^1(F,T)$.
In particular we see that multiplying twice by the order $m$ 
of $F$, $m^2\cdot\Ext^1(F,T)=0$.
This proves the claim.

Now let us consider the exact sequence 
$1\to N\to T\xrightarrow{\times m^2} T\to 1$, 
where $N$ is the kernel of multiplication by $m^2$.
Clearly $N$ is finite and we have an induced exact sequence
\[
\Ext^1(F,N)\to\Ext^1(F,T)\xrightarrow{\times m^2}\Ext^1(F,T)
\]
which shows that the given extension $G$ comes from an extension $F'$ 
of $F$ by $N$.
Then $G$ is the pushout of $F'$ by $N\to T$ and we can 
identify $F'$ with a subgroup of $G$.
\end{proof}

\section{$p$-isogenies}
\label{sect.isogeny}

An isogeny of algebraic groups is a surjective morphism $G\to Q$
with finite kernel. If the kernel is of order prime to
$p$ we say that the isogeny is a $p$-isogeny. 
In this section we will prove
Theorem~\ref{thm:edQuotient} which says that $p$-isogenous
groups have the same essential $p$-dimension. This result will play
a key role in the proof of Theorem~\ref{thm:MainTheorem1}(b)
in Section \ref{sect.MainTheorem1b}.

\begin{thm}\label{thm:edQuotient}
Suppose $G \to Q$ is a $p$-isogeny of algebraic groups over $k$. Then
\begin{enumerate}[label=(\alph*), ref=(\alph*)]
\item\label{edquo1} For any $p$-closed field $K$
containing $k$ the natural map
$H^1(K,G) \to H^1(K,Q)$ is bijective.
\item\label{edquo2} $\ed_k (G;p) = \ed_k (Q;p)$.
\end{enumerate}
\end{thm}

\begin{exa}
Let $E_6^{sc}, E_7^{sc}$ be simply connected simple groups
of type $E_6, E_7$ respectively.
In \cite[9.4, 9.6]{GR} it is shown that if
$k$ is an algebraically closed field of characteristic $\ne 2$ and
$3$ respectively, then
\[ \text{$\ed_k(E_6^{sc}; 2)=3$ and
$\ed_k(E_7^{sc}; 3)=3$.} \]
For the adjoint groups
$E_6^{ad}=E_6^{sc}/\mu_3$, $E_7^{ad}=E_7^{sc}/\mu_2$ we therefore have
% $\ed(E_6^{ad}; 2)=\ed(E_6^{sc}; 2)=3$ and
% $\ed(E_7^{ad}; 3)=\ed(E_7^{sc}; 3)=3$.
\[ \text{$\ed_k(E_6^{ad}; 2)=3$ and
$\ed_k(E_7^{ad}; 3)=3$.} \]
\end{exa}

We will need two lemmas.

\begin{lem}\label{order}
\label{lem:equivalence} Let $N$ be a finite algebraic group over $k$
($\Char k\ne p$). The following are equivalent:
\begin{enumerate}[label=(\alph*), ref=(\alph*)]
\item\label{order1} $p$ does not divide the order of $N$.
\item\label{order2} $p$ does not divide the order of $N(\alg{k})$.
\end{enumerate}
If $N$ is also assumed to be abelian, denote by $N[p]$ the
$p$-torsion subgroup of $N$. The following are equivalent to the
above conditions.
\begin{enumerate}[label=(\alph*$'$), ref=(\alph*$'$)]
\item\label{order3} $N[p](\alg{k})=\{1\}$.
\item\label{order4} $N[p](\pcl{k})=\{1\}$.
\end{enumerate}
\end{lem}

\begin{proof}
\ref{order1}$\iff$\ref{order2}:\; Let $N^\circ$ be the connected
component of $N$ and $N^{et}=N/N^\circ$ the \'etale quotient. Recall
that the order of a finite algebraic group $N$ over $k$ is defined as
$|N|=\dim_kk[N]$ and $|N|=|N^\circ||N^{et}|$, see for example
\cite{Ta}. If $\Char k=0$, $N^\circ$ is trivial, if $\Char k=q\ne p$
is positive, $|N^\circ|$ is a power of $q$. Hence $N$ is of order
prime to $p$ if and only if the \'etale algebraic group $N^{et}$ is.
Since $N^\circ$ is connected and finite, $N^\circ(\alg{k})=\{1\}$ and so $N(\alg{k})$ is of order prime to $p$ if and
only if the group $N^{et}(\alg{k})$ is. Then $|N^{et}|=\dim_k
k[N^{et}]=|N^{et}(\alg{k})|$, cf. \cite[V.29 Corollary]{Bou}.

\ref{order2}$\iff$\ref{order3} $\Rightarrow$ \ref{order4} are clear.

\ref{order3} $\Leftarrow$ \ref{order4}: Suppose $N[p](\alg{k})$ is
nontrivial. The Galois group $\Gamma=\Gal(\alg{k}/\pcl{k})$ is a
pro-$p$ group and acts on the $p$-group $N[p](\alg{k})$. The image
of $\Gamma$ in $\Aut(N[p](\alg{k}))$ is again a (finite) $p$-group
and the size of every $\Gamma$-orbit in $N[p](\alg{k})$ is a power
of $p$. Since $\Gamma$ fixes the identity in $N[p](\alg{k})$, this
is only possible if it also fixes at least $p-1$ more elements. It
follows that $N[p](\pcl{k})$ contains at least $p$ elements, a
contradiction.
\end{proof}
\begin{rem}
Part \ref{order4} could be replaced by the slightly stronger statement that $N[p](\pcl{k}\cap\sep{k})=\{1\}$, but we won't need this in the sequel.
\end{rem}

\begin{lem}\label{copgrp}
Let $\Gamma$ be a profinite group, $G$ an (abstract) finite $\Gamma$-group and $|\Gamma|, |G|$ coprime.
Then $H^1(\Gamma,G)=\{1\}$.
\end{lem}

The case where $\Gamma$ is finite and $G$ abelian is classical. In
the generality we stated, this lemma is also known \cite[I.5, ex. 2]{serre-gc}.

\begin{proof}[Proof of Theorem \ref{thm:edQuotient}]
\ref{edquo1} Let $N$ be the kernel of $G\to Q$ and $K=\pcl{K}$ be a $p$-closed field over $k$.
Since $\sep{K}=\alg{K}$ (see Proposition~\ref{ptopprop}(b)), the sequence of $\sep{K}$-points
$1\to N(\sep{K})\to G(\sep{K})\to Q(\sep{K})\to 1$ is exact.
By Lemma~\ref{order}, the order of $N(\sep{K})$ is not divisible by $p$ and therefore coprime to the order of
$\Psi=\Gal(\sep{K}/K)$.
Thus $H^1(K,N)=\{1\}$ (Lemma~\ref{copgrp}).
Similarly, if $_cN$ is the group $N$ twisted by a cocycle $c:\Psi\to G$, $_cN(\sep{K})=N(\sep{K})$ is of order prime to $p$ and $H^1(K,\,_cN)=\{1\}$.
It follows that $H^1(K,G)\to H^1(K,Q)$ is injective, cf. \cite[I.5.5]{serre-gc}.

Surjectivity is a consequence of \cite[I. Proposition 46]{serre-gc}
and the fact that the $q$-cohomological dimension of $\Psi$
is $0$ for any divisor $q$ of $|N(\sep{K})|$ (Proposition \ref{ptopprop}).

This concludes the proof of part~\ref{edquo1}. Part~\ref{edquo2} immediately follows
from (a) and Proposition~\ref{prop:bijectionPrimeToPClosure}.
\end{proof}

\section{Proof of Theorem~\ref{thm:MainTheorem1}(b)}
\label{sect.MainTheorem1b}

Let $k$ be a closed field and $G = T \times F$,
where $T$ is a torus and $F$ is a finite $p$-group, defined over $k$. 
Our goal is to show that
\begin{equation} \label{e.upper-bound1}
\ed_k(G;p) \le \dim(\rho) - \dim G \, ,
\end{equation}
where $\rho$ is a $p$-faithful representation
of $G$ defined over $k$.

\begin{lem} \label{lem.rigid}
If a representation $\rho \colon G \to \GL (V)$ 
is $p$-faithful, then $G/\ker(\rho) \to \GL(V)$ 
is generically free. In other words, $\rho$ is
$p$-generically free.
\end{lem}

\begin{proof}
Since $\ker(\rho)$ has order prime to $p$, its image under 
the projection map $G = T\times F \to F$ is trivial. 
Hence $\ker(\rho) \subset T$ and $T/N$ is
again a torus. So without loss of generality, we may
assume $\rho$ is faithful.

Let $V_1\subsetneq V$ be a closed subset of $V$ such that $T$ acts freely 
on $V \setminus V_1$.  Let $n = p^r$ be the order of $F$ and $V_2$ be the (finite) union of the fixed point sets of $1 \ne g \in T[n] \times F$. Here as usual, $T[n]$ denotes the $n$-torsion subgroup of $T$. Since $\rho$ is faithful none of these fixed point sets are all of $V$, hence $U\colonequals V \setminus (V_1\cup V_2)$ is a dense open subset of $V$.

We claim that $\Stab_G(v) = \{ 1 \}$ for every $v \in U$.
Indeed, assume $1 \ne g = (t, f) \in \Stab_G(v)$. Since $v \not \in V_2$,
$t^n \ne 1$. Then $1 \ne g^n = (t^n, 1)$ lies in both $T$ and $\Stab_G(v)$.
Since $v \not \in V_1$, this is a contradiction.
\end{proof}

Now suppose $\rho$ is any $p$-faithful representation of $G$. 
Then~\eqref{e.upper-bound} yields
\[ \ed_k (G/N;p)\leq \dim (\rho) - \dim (G/\ker(\rho)) = 
\dim (\rho) - \dim (G) \, . \]
By Theorem~\ref{thm:edQuotient} 
\[ \ed_k(G;p) = \ed(G/N; p) \leq \dim(\rho) - \dim (G) \, , \]
as desired. This completes the proof of~\eqref{e.upper-bound1} and
thus of Theorem~\ref{thm:MainTheorem1}(b).
\qed

\begin{cor}\label{cor:Sylow}
Let $G$ be a finite algebraic
group over a $p$-closed field $k=\pcl{k}$. Then $G$ has a Sylow-$p$
subgroup $G_p$ defined over $k$ and
\[
\ed_k(G;p)=\ed_k(G_p;p)=\ed_k(G_p)=\min \dim (\rho)
\]
where the minimum is taken over all faithful representations of
$G_p$ over $k$.
\end{cor}
\begin{proof}
By assumption,  $\Gamma=\Gal(\sep{k}/k)$ is a pro-$p$ group. It acts
on the set of Sylow-$p$ subgroups of $G(\sep{k})$. Since the number
of such subgroups is prime to $p$, $\Gamma$ fixes at least one of them and by
Galois descent one obtains a subgroup $G_p$ of $G$. By
Lemma~\ref{order}, $G_p$ is a Sylow-$p$ subgroup of $G$. The first
equality $\ed_k(G;p)=\ed_k(G_p;p)$ is shown in \cite[4.1]{mr1} (the reference is for smooth groups but can be generalized to the non-smooth case as well). The
minimal $G_p$-representation $\rho$ from
Theorem~\ref{thm:MainTheorem1}(b) is faithful and thus $\ed_k(G_p)\le
\dim (\rho)$, see for example \cite[Prop. 4.11]{BF}. The Corollary
follows.
\end{proof}
\begin{rem}
\label{rem.nonIsomorphicSylowSubgroups}
Two Sylow-$p$ subgroups of $G$ defined over $k=\pcl{k}$ do not need to be isomorphic over $k$.
\end{rem}
%%%%%%%%%%%%%%%%%%%%%%%%%%%%%%%%%%%%%%%%%%%%%%%%%%%%

\section{An additivity theorem}
\label{sect.additivity}

The purpose of this section is to prove the following: 

\begin{thm} \label{thm.additive}
Let $G_1$ and $G_2$ be direct products of tori and $p$-groups over a field $k$.
Then $\ed_k(G_1 \times G_2; p) = \ed_k(G_1; p) + \ed_k(G_2; p)$.
\end{thm} 

Let $G$ be an algebraic group defined over $k$ and $C$ be a
$k$-subgroup of $G$. Denote the minimal dimension of a
representation $\rho$ of $G$ (defined over $k$) such that $\rho_{|
\, C}$ is faithful by $f(G, C)$.

\begin{lem} \label{lem.additive}
For $i = 1, 2$ let $G_i$ be an algebraic group defined over $k$ and
$C_i$ be a central $k$-subgroup of $G_i$. Assume that $C_i$ is
isomorphic to $\mu_p^{r_i}$ over $k$ for some $r_1, r_2 \ge 0$. Then
\[ f(G_1 \times G_1; C_1 \times C_2) = f(G_1; C_1) + f(G_2; C_2) \,
. \]
\end{lem}

Our argument is a variant of the proof of~\cite[Theorem 5.1]{KM},
where $G$ is assumed to be a (constant) finite $p$-group and $C$ is
the socle of $G$.

\begin{proof} For $i = 1, 2$ let
$\pi_i \colon G_1 \times G_2 \to G_i$ be the natural projection and
$\epsilon_i \colon G_i \to G_1 \times G_2$ be the natural inclusion.

If $\rho_i$ is a $d_i$-dimensional $k$-representation of $G_i$ whose
restriction to $C_i$ is faithful, then clearly $\rho_1 \circ \pi_1
\oplus \rho_2 \circ \pi_2$ is a $d_1 + d_2$-dimensional
representation of $G_1 \times G_2$ whose restriction to $C_1 \times
C_2$ is faithful.  This shows that
\[ f(G_1 \times G_1; C_1 \times C_2) \le f(G_1; C_1) + f(G_2; C_2) \, . \]
To prove the opposite inequality, let $\rho \colon G_1 \times G_2
\to \GL(V)$ be a $k$-representation such that $\rho_{| \, C_1 \times
C_2}$ is faithful, and of minimal dimension
\[ d = f(G_1 \times G_1; C_1 \times C_2) \]
with this property. Let $\rho_1,\rho_2,\dotsc,\rho_n$ denote the
irreducible decomposition factors in a decomposition series of
$\rho$. Since $C_1 \times C_2$ is central in $G_1 \times G_2$, each
$\rho_i$ restricts to a multiplicative character of $C_1 \times C_2$
which we will denote by $\chi_i$. Moreover since $C_1 \times
C_2\simeq \mu_p^{r_1+r_2}$ is linearly reductive $\rho_{| \, C_1
\times C_2}$ is a direct sum $\chi_1^{\oplus d_1} \oplus \dotsb
\oplus \chi_n^{\oplus d_n}$ where $d_i=\dim V_i$. It is easy to see
that the following conditions are equivalent:

\smallskip
(i) $\rho_{| \, C_1 \times C_2}$ is faithful,

\smallskip
(ii) $\chi_1, \dots, \chi_n$ generate $(C_1 \times C_2)^*$ as an
abelian group.

\smallskip
\noindent In particular we may assume that $\rho=\rho_1\oplus \dotsb
\oplus \rho_n$. Since $C_i$ is isomorphic to $\mu_p^{r_i}$, we will
think of $(C_1 \times C_2)^*$ as a $\FF_p$-vector space of dimension
$r_1$ + $r_2$. Since (i) $\Leftrightarrow$ (ii) above, we know that
$\chi_1, \dots, \chi_n$ span $(C_1 \times C_2)^*$. In fact, they
form a basis of $(C_1 \times C_2)^*$, i.e., $n = r_1 + r_2$. Indeed,
if they were not linearly independent we would be able to drop some
of the terms in the irreducible decomposition $\rho_1 \oplus \dots
\oplus \rho_n$, so that the restriction of the resulting
representation to $C_1 \times C_2$ would still be faithful,
contradicting the minimality of $\dim(\rho)$.

We claim that it is always possible to replace each $\rho_j$ by
$\rho_j'$, where $\rho_j'$ is either $\rho_j \circ \epsilon_1 \circ
\pi_1$ or $\rho_j \circ \epsilon_2 \circ \pi_2$ such that the
restriction of the resulting representation $\rho' = \rho_1' \oplus
\dots \oplus \rho_n'$ to $C_1 \times C_2$ remains faithful.  Since
$\dim(\rho_i) = \dim(\rho_i')$, we see that $\dim(\rho') =
\dim(\rho)$. Moreover, $\rho'$ will then be of the form $\alpha_1
\circ \pi_1 \oplus \alpha_2 \circ \pi_2$, where $\alpha_i$ is a
representation of $G_i$ whose restriction to $C_i$ is faithful.
Thus, if we can prove the above claim, we will have
\begin{align*}
f(G_1 \times G_1; C_1 \times C_2) & = \dim(\rho) = \dim(\rho') =
\dim(\alpha_1) + \dim(\alpha_2)  \\
 &\ge f(G_1, C_1) + f(G_2, C_2) \, ,
\end{align*}
as desired.

To prove the claim, we will define $\rho_j'$ recursively for $j = 1,
\dots, n$. Suppose $\rho_1', \dots, \rho_{j-1}'$ have already be
defined, so that the restriction of
\[ \rho_1' \oplus \dots \oplus \rho_{j-1}' \oplus \rho_j \dots \oplus \rho_n \]
to $C_1 \times C_2$ is faithful. For notational simplicity, we will
assume that $\rho_1 = \rho_1', \dots, \rho_{j-1} = \rho_{j-1}'$.
Note that
\[ \chi_j = (\chi_j \circ \epsilon_1 \circ \pi_1) +
(\chi_j \circ \epsilon_2 \circ \pi_2) \, . \] Since $\chi_1, \dots,
\chi_n$ form a basis $(C_1 \times C_2)^*$ as an $\FF_p$-vector
space, we see that (a) $\chi_j \circ \epsilon_1 \circ \pi_1$
or (b) $\chi_j \circ \epsilon_2 \circ \pi_2$  does not lie in
$\Span_{\FF_p}(\chi_1, \dots, \chi_{j-1}, \chi_{j+1}, \dots,
\chi_n)$. Set
\[ \rho_j' \colonequals \begin{cases} 
\text{$\rho_j \circ \epsilon_1 \circ \pi_1$ in case (a), and} \\
\text{$ \rho_j \circ \epsilon_2 \circ \pi_2$, otherwise.}
\end{cases} \]
Using the equivalence of (i) and (ii) above, we see
that the restriction of
\[ \rho_1 \oplus \dots \oplus \rho_{j-1} \oplus
\rho_j' \oplus \rho_{j+1}, \dots \oplus \rho_n \] to $C$ is
faithful. This completes the proof of the claim and thus of
Lemma~\ref{lem.additive}.
\end{proof}

\begin{proof}[Proof of Theorem~\ref{thm.additive}]
We can pass to a $p$-closure $\pcl{k}$ by
Lemma~\ref{lem:edPrimeToPClosure}. 
Let $C(G)$ be as in Definition~\ref{def.C(G)}.
% In Section~\ref{sect.C(G)} we have associated to
% each group $G$, which is a direct product of a twisted $p$-group and
% a group of multiplicative type, a central subgroup $C(G)$ over $k$
% such that $C(G)$ is a power of $\mu_p$ and
By Theorem~\ref{thm:MainTheorem1}(b)
\[ \ed(G; p) = f(G, C(G)) - \dim G \, ; \]
cf. Corollary~\ref{cor.C(G)}.
%Moreover, Lemma~\ref{lem3.1}\ref{SplitAd} says that
%\[ C(G_1 \times G_2) = C(G_1) \times C(G_2) \, . \]
Furthermore, we have $C(G_1 \times G_2) = C(G_1) \times C(G_2)$;
cf.~Lemma~\ref{lem3.1}(d).  Applying Lemma~\ref{lem.additive} 
finishes the proof.
\end{proof}

%%%%%%%%%%%%%%%%%%%%%%%%%%%%%%%%%%%

\section{Modules and lattices} 
\label{sect.modules+lattices}
In this section we rewrite the value of $\ed_k(G;p)$ in terms of the
character module $X(G)$ for an \textit{abelian} group $G$ which is an extension of a $p$-group and a torus. 
Moreover we show that tori with locally isomorphic character lattices have the same essential dimension. 
We need the following preliminaries. 

Let $R$ be a commutative ring
(we use $R=\ZZ$ and $R=\ZZ_{(p)}$ mostly) and $A$ an $R$-algebra. An
$A$-module is called an {\em $A$-lattice} if it is finitely
generated and projective as an $R$-module. For $A=\ZZ\Gamma$
($\Gamma$ a group) this is as usual a free abelian group of finite
rank with an action of $\Gamma$. Particular cases of
$R\Gamma$-lattices are {\em permutation lattices} $L=R[\Lambda]$
where $\Lambda$ is a $\Gamma$-set. \par For $\Gamma=\Gal(\sep{k}/k)$
the absolute Galois group of $k$ we tacitly assume that our
$R\Gamma$-lattices are continuous, i.e.~$\Gamma$ acts through a
finite quotient $\overline{\Gamma}$. Under the anti-equivalence
$\Diag$ a $\ZZ\Gamma$-lattice corresponds to an algebraic $k$-torus.
A torus $S$ is called {\em quasi split} if it corresponds to a
permutation lattice. Equivalently $S\simeq R_{E/k}(\Gm)$ where $E/k$
is \'etale and $R_{E/k}$ denotes Weil restriction. \par \smallskip
Recall that $\ZZ_{(p)}$ denotes the localization of the ring $\ZZ$
at the prime ideal $(p)$. For a $\ZZ$-module $M$ we also write
$M_{(p)}\colonequals \ZZ_{(p)}\otimes M$. \par

When $\Gamma=\Gal(\sep{k}/k)$ we will often pass from
$\ZZ\Gamma$-lattices to $\ZZ_{(p)}\Gamma$-lattices. This corresponds
to identifying $p$-isogeneous tori:
\begin{lem}
\label{lem.pisogeny}
Let $\Gamma=\Gal(\sep{k}/k)$ and let $M,L$ be $\ZZ\Gamma$-lattices. Then the following statements are equivalent:
\begin{enumerate}[label=(\alph*), ref=(\alph*)]
\item \label{equivA} $L_{(p)}\simeq M_{(p)}$.
\item \label{equivB} There exists an injective map $\phi\colon L \to M$ of $\ZZ\Gamma$-modules with cokernel $Q$ finite of order prime to $p$.
\item \label{equivC} There exists a $p$-isogeny $\Diag(M) \to \Diag(L)$.
\end{enumerate}
\end{lem}
\begin{proof}
The equivalence \ref{equivB} $\Leftrightarrow$ \ref{equivC} is clear
from the anti-equivalence of $\Diag$.

The implication \ref{equivB} $\Rightarrow$ \ref{equivA} follows from
$Q_{(p)}=0$ and that tensoring with $\ZZ_{(p)}$ is exact.

For the implication \ref{equivA} $\Rightarrow$ \ref{equivB} we use that $L$ and $M$ can be considered as subsets of $L_{(p)}$ (resp.~$M_{(p)}$). 
The image of $L$ under a map $\alpha \colon L_{(p)} \to M_{(p)}$ of $\ZZ_{(p)}\Gamma$-modules lands in $\frac{1}{m} M$ for some $m\in \NN$ (prime to $p$) and the index of $\alpha(L)$ in $\frac{1}{m}M$ is finite and prime to $p$ if $\alpha$ is surjective. 
Since $\frac{1}{m}M\simeq M$ as $\ZZ\Gamma$-modules the claim follows.
\end{proof}
\begin{cor} \label{cor.lattice}
Let $G$ be an abelian group which is an extension of a $p$-group by a torus over $k$
and $\Gamma \colonequals \Gal(\sep{k}/k)$ be the absolute Galois group of $k=\pcl{k}$. Let $\Gamma$ act through a finite quotient $\overline{\Gamma}$ on $X(G)$.
Then
 \[ \ed_k(G;p) = \min \rk L - \dim G \, , \]
where the minimum is taken over all permutation $\ZZ\overline{\Gamma}$-lattices
$L$ which admit a map of $\ZZ\overline{\Gamma}$-modules to $X(G)$ with cokernel
finite of order prime to $p$. \par \smallskip If $G$ is a torus,
then the minimum can also be taken over all
$\ZZ_{(p)}\overline{\Gamma}$-lattices $L$ which admit a surjective map of
$\ZZ_{(p)}\overline{\Gamma}$-modules to $X(G)_{(p)}$.
\end{cor}
\begin{proof}
Let us prove the first claim. In view of Theorem
\ref{thm:MainTheorem1}(a) it suffices to show that the least dimension
of a $p$-faithful representation of $G_{\pcl{k}}$ over $\pcl{k}$ is
equal to the least rank of a permutation $\ZZ\overline{\Gamma}$-module $L$
which admits a map to $X(G)$ with cokernel finite of order prime to
$p$.

Assume we have such a map $L\to X(G)$. Using the anti-equivalence
$\Diag$ we obtain a $p$-isogeny $G \to \Diag(L)$. We can embed the
quasi-split torus $\Diag(L)$ in $\GL_n$ where $n=\rk L$
\cite[Section 6.1]{voskresenskii}. This yields a $p$-faithful
representation of $G$ of dimension $\rk L$.

Conversely let $\rho \colon G \to \GL(V)$ be a $p$-faithful
representation of $G$. Since $\sep{G}$ is diagonalizable, there
exist characters $\chi_1,\dotsc,\chi_n \in X(G)$ such that $G$ acts
on $\sep{V}$ via diagonal matrices with entries
$\chi_1(g),\dotsc,\chi_n(g)$ (for $g \in G$) with respect to a
suitable basis of $\sep{V}$. Moreover $\overline{\Gamma}$ permutes the set
$\Lambda \colonequals \{\chi_1,\dotsc,\chi_n\}$. Define a map $\phi
\colon \ZZ[\Lambda] \to X(G)$ of $\ZZ\overline{\Gamma}$-modules by sending the
basis element $\chi_i\in \Lambda$ of $L\colonequals \ZZ[\Lambda]$ to
itself. Then the $p$-faithfulness of $\rho$ implies that the
cokernel of $\phi$ is finite and of order prime to $p$. Moreover
$\rk L = |\Lambda| \leq n= \dim V$.

Now consider the case
where $G$ is a torus. Assume we have a surjective map $\alpha \colon
L \to X(G)_{(p)}$ of $\ZZ_{(p)}\overline{\Gamma}$-modules where
$L=\ZZ_{(p)}[\Lambda]$ is permutation, $\Lambda$ a $\overline{\Gamma}$-set.
Then $\alpha(\Lambda)\subseteq \frac{1}{m}X(G)$ for some $m \in \NN$
prime to $p$ (note that $\frac{1}{m}X(G)$ can be considered as a
subset of $X(G)_{(p)}$ since $X(G)$ is torsion free). By
construction the induced map $\ZZ[\Lambda] \to \frac{1}{m}X(G)
\simeq X(G)$ becomes surjective after localization at $p$, hence its
cokernel is finite of order prime to $p$.
\end{proof}
\begin{cor} \label{cor.cyclic}
Let $A$ be a finite (twisted) cyclic $p$-group over $k$. Let $l/k$
be a minimal Galois splitting field of $A$, and $\Gamma \colonequals
\Gal(l/k)$. Then \[ \ed(A;p)=|\Gamma|.\]
\end{cor}

\begin{proof} Since $[l:k]$ is a power of $p$, $\pcl{l}/\pcl{k}$ is
a Galois extension of the same degree and the same Galois group as
$l/k$. So we can assume $k=\pcl{k}$.

By Corollary~\ref{cor.lattice} $\ed(A;p)$ is equal to the least
cardinality of a $\Gamma$-set $\Lambda$ such that there exists a map
$\phi \colon \ZZ[\Lambda]\to X(A)$ of $\ZZ\Gamma$-modules with
cokernel finite of order prime to $p$. The group $X(A)$ is a
(cyclic) $p$-group, hence $\phi$ must be surjective. Moreover
$\Gamma$ acts faithfully on $X(A)$. Surjectivity of $\phi$ implies
that some element $\lambda \in \Lambda$ maps to a generator $a$ of
$X(A)$. Hence $|\Lambda|\geq|\Gamma \lambda|\geq |\Gamma a| =
|\Gamma|$. Conversely we have a surjective homomorphism $\ZZ[\Gamma
a] \to X(A)$ that sends $a$ to itself. Hence the claim follows.

%Since $[l:k]$ is a power of $p$, $\pcl{l}/\pcl{k}$ is
%a Galois extension of the same degree and the same Galois group as $l/k$.
%The Weil restriction $R_{\pcl{l}/\pcl{k}}(\mathbb{A}^1)$
%is a $|\Gal(L/k)|$-dimensional minimal faithful representation
%of $A$ defined over $\pcl{k}$. Thus $\ed(A; p) \le |\Gamma|$.
%The converse follows from Corollary~\ref{cor.lattice}.
\end{proof}

\begin{rem} \label{rem.cyclic}
In the case of twisted cyclic
groups of order $4$ Corollary~\ref{cor.cyclic} is due to Rost~\cite{Ro}
(see also~\cite[Theorem 7.6]{BF}),
and in the case of cyclic groups of order $8$ to Bayarmagnai~\cite{Ba}.
The case of constant groups of arbitrary prime power order
is due to Florence~\cite{Fl}; it is now a special case of
the Karpenko-Merkurjev Theorem~\ref{thm.km}.
\end{rem}

\section{Proof of Theorem~\ref{thm:MainTheorem1}(c)}
\label{sect.MainTheorem1c}

We will prove Theorem~\ref{thm:MainTheorem1}(c) by using 
the lattice point of view from Section~\ref{sect.modules+lattices} 
and the additivity theorem from Section~\ref{sect.additivity}.

Let $\overline{\Gamma}$ be a finite group. Two $\ZZ\overline{\Gamma}$-lattices $M,N$ are said to be in the same {\em genus} if $M_{(p)} \simeq N_{(p)}$ for all primes $p$, cf. \cite[31A]{CR}. 
It is sufficient to check this condition for divisors $p$ of the order of $\overline{\Gamma}$. 
By a theorem of A.V.~Ro\v{\i}ter \cite[Theorem 31.28]{CR} $M$ and $N$ are in the same genus if and only if there exists a $\ZZ\overline{\Gamma}$-lattice $L$ in the genus of the free $\ZZ \overline{\Gamma}$-lattice of rank one such that $M \oplus \ZZ \overline{\Gamma} \simeq N \oplus L$. 
This has the following consequence for essential dimension:
\begin{prop}
\label{prop.equalED}
Let $T,T'$ be $k$-tori. If the lattices $X(T), X(T')$ belong to the same genus then \[\ed_k(T)=\ed_k(T) \text{ and } \ed_k(T;\ell)=\ed_k(T;\ell) \text{ for all primes } \ell.\] 
\end{prop}
\begin{proof}
Let $\Gal(\sep{k}/k)$ act through a finite quotient $\overline{\Gamma}$ on $X(T)$ and $X(T')$. 
By assumption there exists a $\ZZ \overline{\Gamma}$-lattice $L$ in the genus of $\ZZ \overline{\Gamma}$ such that $X(T)\oplus \ZZ \overline{\Gamma} \simeq X(T')\oplus L$. 
The torus $S=\Diag(\ZZ \overline{\Gamma})$ has a generically free representation of dimension $\dim S$, hence $\ed_k(S)=0$. 
Since $L$ is a direct summand of $\ZZ\overline{\Gamma} \oplus \ZZ\overline{\Gamma}$ the torus $S'\colonequals \Diag(L)$ has $\ed_k(S')\leq \ed_k(S\times S)\leq 0$ as well, where the first inequality follows from \cite[Remarks 1.16 (b)]{BF}. 
Therefore \[\ed_k(T)\leq \ed_k(T\times S)=\ed_k(T'\times S')\leq \ed_k(T')+\ed_k(S') = \ed_k(T')\] and similarly $\ed_k(T')\leq \ed_k(T)$. Hence $\ed_k(T)=\ed_k(T')$. \par
A similar argument shows that $\ed_k(T;\ell)=\ed_k(T';\ell)$ for any prime $\ell$. 
This concludes the proof.
\end{proof}

\begin{cor}
\label{cor.absoluteED}
Let $k=\pcl{k}$ be a $p$-closed field and $T$ a $k$-torus. Then
\[\ed_k(T)=\ed_k(T;p)=\min \dim(\rho) - \dim T,\] where the minimum is taken over all $p$-faithful representations of $T$.
\end{cor}
\begin{proof}
The second equality follows from Theorem \ref{thm:MainTheorem1}(a) and the inequality $\ed_k(T;p)\leq \ed_k(T)$ is clear. Hence it suffices to show $\ed_k(T)\leq \ed_k(T;p)$. 
Let $\rho \colon T \to \GL(V)$ be a $p$-faithful representation of minimal dimension so that $\ed_k(T;p)=\dim \rho - \dim T$. 
The representation $\rho$ can be considered as a faithful representation of the torus $T'= T/N$ where $N\colonequals \ker \rho$ is finite of order prime to $p$. 
By construction the character lattices $X(T)$ and $X(T')$ are isomorphic after localization at $p$. 
Since $\Gal(\sep{k}/k)$ is a (profinite) $p$-group it follows that $X(T)$ and $X(T')$ belong to the same genus. Hence by Proposition \ref{prop.equalED} we have $\ed_k(T')=\ed_k(T)$. 
Moreover $\ed_k(T')\leq \dim \rho - \dim T'$, since $\rho$ is a generically free representation of $T'$. 
This finishes the proof.
\end{proof}

\begin{proof}[Proof of Theorem~\ref{thm:MainTheorem1}(b)]
The equality $\ed_{\pcl{k}}(G_{\pcl{k}};p)=\ed_k(G;p)$ follows from Lemma \ref{lem:edPrimeToPClosure}. 
Now we are assuming $G=T \times F$ for a torus $T$ and a $p$-group $F$ over $k$, which is $p$-closed. 
Notice that a minimal $p$-faithful representation of $F$ from Theorem~\ref{thm:MainTheorem1}(a) is also faithful, and therefore $\ed_k(F;p)=\ed_k(F)$. 
Combining this with Corollary~\ref{cor.absoluteED} and the additivity Theorem~\ref{thm.additive}, we see 
\[ \ed(T\times F) \leq \ed(T) + \ed(F) = \ed(T;p) + \ed(F;p) = \ed(T\times F;p) \leq \ed(T \times F). \]
This completes the proof.
\end{proof}

\begin{rem} \label{rem.p-faithful}
The following example shows that ``$p$-faithful" cannot be replaced
by ``faithful" in the statement of
Theorem~\ref{thm:MainTheorem1}(a) (and Corollary~\ref{cor.absoluteED}), even in the case where $G$ is a
torus.

Let $p$ be a prime number such that the ideal class group of
$\QQ(\zeta_p)$ is non-trivial (this applies to all but finitely many
primes, e.g.~to $p=23$). This means that the subring $R=\ZZ[\zeta_p]
\subseteq \QQ(\zeta_p)$ of algebraic integers has non-principal
ideals. Let $k$ be a field which admits a Galois extension $l$ of
degree $p$ and let $\Gamma\colonequals \Gal(\sep{k}/k)$,
$\overline{\Gamma}\colonequals \Gal(l/k)\simeq \Gamma/\Gamma_l
\simeq C_p$ where $\Gamma_l = \Gal(\sep{k}/l)$ and $C_p$ denotes the cyclic group of order $p$. \par We endow the
ring $R$ with a $\ZZ \Gamma$-module structure through the quotient
map $\Gamma \to \overline{\Gamma}$ by letting a generator of
$\overline{\Gamma}$ act on $R$ via multiplication by $\zeta_p$. The
$k$-torus $Q\colonequals \Diag(R)$ is isomorphic to the Weil
restriction $R_{l/k}(\Gm)$ and has a $p$-dimensional faithful
representation. We will construct a $k$-torus $G$ with a $p$-isogeny
$G\to Q$, such that $G$ does not have a $p$-dimensional faithful
representation. \par Let $I$ be a non-principal ideal of $R$. We may
consider $I$ as a $\ZZ \Gamma$-module and set $G\colonequals
\Diag(I)$. We first show that $I$ and $R$ become isomorphic as
$\ZZ\Gamma$-modules after localization at $p$. For this purpose let
$I^\ast=\{x \in \QQ(\zeta_p) \mid xI\subseteq R\}$ denote the
inverse fractional ideal. We have $I\oplus I^\ast \simeq R \oplus R$
by \cite[Theorem 34.31]{CR}. The Krull-Schmidt Theorem \cite[Theorem
36.1]{CR} for $\ZZ_{(p)} C_p$-lattices implies $I_{(p)} \simeq
R_{(p)}$, hence the claim. Therefore by Lemma \ref{lem.pisogeny}
there exists a $p$-isogeny $G\to Q$, which shows in particular that
$G$ has a $p$-faithful representation of dimension $p$. \par Assume
that $G$ has a $p$-dimensional faithful representation. Similarly as
in the proof of Corollary \ref{cor.lattice} this would imply the
existence of a surjective map of $\ZZ\Gamma$-lattices
$\ZZ\overline{\Gamma}\to I$. However such a map cannot exist since
$I$ is non-principal, hence non-cyclic as a $\ZZ\Gamma$-module.
\end{rem}

\section{Tori of essential dimension $\le 1$}
\label{sect.ed<=1}

\label{subsec:SpecialTori}

\begin{thm} \label{thm.ed=0}
Let $T$ be a torus over $k$, $\pcl{k}$ a $p$-closure and $\Gamma=\Gal(\alg{k}/\pcl{k})$.
The following are equivalent:
\begin{enumerate}[label=(\alph*), ref=(\alph*)]
\label{prop:pSpecialTori}
\item\label{pSpecial1} $\ed_k(T;p)=0$.
\item\label{pSpecial5} $\ed_{\pcl{k}}(T;p)=0$.
\item\label{pSpecial8} $\ed_{\pcl{k}}(T)=0$
%\item\label{pSpecial2} $p$ does not divide the degree of the splitting
\item\label{pSpecial6} $H^1(K,T)=\{1\}$ for any $p$-closed field $K$ 
containing $k$.
\item\label{pSpecial3} $X(T)_{(p)}$ is a $\ZZ_{(p)}\Gamma$-permutation module.
\item\label{pSpecial7} $X(T)$ is an invertible $\ZZ\Gamma$-lattice 
(i.e a direct summand of a permutation lattice).
\item\label{pSpecial4} There is a torus $S$ over $\pcl{k}$ and an isomorphism
\[ T_{\pcl{k}}\times S\simeq \R_{E/\pcl{k}}(\Gm), \]
for some \'etale algebra $E$ over $\pcl{k}$.
\end{enumerate}
\end{thm}

\begin{rem}
A prime $p$ for which any of these statements fails is 
called a {\em torsion prime} of $T$.

%In part \ref{pSpecial4}, $\R_{L/K}$ denotes Weil restriction.
%A torus of the form $\R_{E/K}(\Gm)$ is called {\em quasi split}.
\end{rem}

\begin{proof}
\ref{pSpecial1} $\Leftrightarrow$ \ref{pSpecial5} is Lemma~\ref{lem:edPrimeToPClosure}.

\ref{pSpecial1} $\Leftrightarrow$ \ref{pSpecial6} follows from \cite[Proposition 4.4]{Me1}.

\ref{pSpecial8} $\Rightarrow$ \ref{pSpecial5} is clear.

\ref{pSpecial5} $\Rightarrow$ \ref{pSpecial3}: This follows
from Corollary~\ref{cor.lattice}.
Indeed, $\ed_k(T;p)=0$ implies the existence of a $\ZZ_{(p)}\Gamma$-permutation lattice $L$ together with a surjective homomorphism $\alpha:L\to X(T)_{(p)}$ and $\rk L = \rk X(T)_{(p)}$.
It follows that $\alpha$ is injective and $X(T)_{(p)}\simeq L$.

\ref{pSpecial3} $\Rightarrow$ \ref{pSpecial7}: Let $L$ be a $\ZZ\Gamma$-permutation lattice such that $L_{(p)}\simeq X(T)_{(p)}$.
Then by \cite[Corollary 31.7]{CR} there is 
a $\ZZ\Gamma$-lattice $L^\prime$ such that 
$L\oplus L\simeq X(T)\oplus L^\prime$.

\ref{pSpecial4} $\Rightarrow$ \ref{pSpecial8}:
The torus $R=\R_{E/\pcl{k}}(\Gm)$ has a faithful representation 
of dimension $\dim R$ (over $\pcl{k}$) and hence $\ed_{\pcl{k}}(R) = 0$.
Since $T_{\pcl{k}}$ is a direct factor of $R$ we must have 
$\ed_{\pcl{k}}(T)\le 0$ by \cite[Remarks 1.16 b)]{BF}.

\ref{pSpecial7} $\Leftrightarrow$ \ref{pSpecial4}: A permutation 
lattice $P$ can be written as
\[ P=\bigoplus_{i+1}^m\ZZ[\Gamma/\Gamma_{L_i}],\]
for some (separable) extensions $L_i/\pcl{k}$ and
$\Gamma_{L_i}=\Gal(\alg{k}/L_i)$. Set $E=L_1\times\cdots\times L_m$.
The torus corresponding to $P$ is exactly $\R_{E/\pcl{k}}(\Gm)$, cf.
\cite[3. Example 19]{voskresenskii}.
\end{proof}

\begin{exa}
Let $T$ be a torus over $k$ of rank $<p-1$. 
Then $\ed_k(T;p)=0$.
This follows from the fact that there is no non-trivial integral representation of dimension $<p-1$ of any $p$-group, see for example \cite[Satz]{AP}.
Thus any finite quotient of $\Gamma=\Gal(\alg{k}/\pcl{k})$ acts trivially on $X(T)$ and so does $\Gamma$.
\end{exa}

\begin{rem}
The equivalence of parts \ref{pSpecial6} 
and \ref{pSpecial7} in Theorem~\ref{thm.ed=0}
can also be deduced from~\cite[Proposition 7.4]{CTS}.
\end{rem}

\begin{thm}\label{thm.ed=1}

Let $p$ be an odd prime, $T$ an algebraic torus over $k$, and $\Gamma = \Gal(\alg{k}/k^{(p)})$. \hfill
\begin{enumerate}[label=(\alph*), ref=(\alph*)]
\item \label{ed1a} $\ed(T;p) \leq 1$ iff there exists a $\Gamma$-set $\Lambda$ and an $m\in \ZZ[\Lambda]$ fixed by $\Gamma$ such that $X(T)_{(p)} \cong \ZZ_{(p)}[\Lambda] / \langle m \rangle $ as $\ZZ_{(p)} \Gamma$-lattices.
\item \label{ed1b} $\ed(T;p)=1$ iff $m=\sum a_{\lambda} \lambda$ from part \ref{ed1a} is not $0$ and for any $\lambda \in \Lambda$ fixed by $\Gamma$, $a_{\lambda}=0 \mod p$.
\item \label{ed1c} If $\ed(T;p)=1$ then $T_{\pcl{k}}\cong T' \times S$ where $\ed_{\pcl{k}}(S;p)=0$ and $X(T')_{(p)}$ is an indecomposable $\ZZ_{(p)} \Gamma$-lattice, and $\ed_{\pcl{k}}(T';p)=1$.
\end{enumerate}

\end{thm}

\begin{proof}
\ref{ed1a} If $\ed(T;p)=1$, then by Corollary~\ref{cor.lattice}
there is a map of $\ZZ \Gamma$-lattices
from $\ZZ[\Lambda]$ to $X(T)$ which becomes surjective after 
localization at $p$ and whose kernel
is generated by one element. Since the kernel is stable
under $\Gamma$, any element of $\Gamma$ sends
a generator $m$ to either itself or its negative.
Since $p$ is odd, $m$ must be fixed by $\Gamma$.

The $\ed(T;p)=0$ case and the converse follows
from Theorem~\ref{thm:MainTheorem2} or Corollary~\ref{cor.lattice}.

\ref{ed1b} Assume we are in the situation of \ref{ed1a},
and say $\lambda_0 \in \Lambda$ is fixed by $\Gamma$
and $a_{\lambda_0}$ is not $0 \mod p$. Then
$X(T)_{(p)} \cong \ZZ_{(p)} [\Lambda - \{ \lambda_0 \}]$,
so by Theorem \ref{thm.ed=0} we have $\ed(T;p)=0$.

Conversely, assume $\ed(T;p)=0$. 
Then by Theorem \ref{thm.ed=0}, we have an exact sequence
% \begin{equation} \label{e.11.1}
$ 0 \to \langle m \rangle \to \ZZ_{(p)}[\Lambda] \to 
\ZZ_{(p)}[\Lambda'] \to 0 $  
% \end{equation}
for some $\Gamma$-set $\Lambda'$ with one fewer element than $\Lambda$.
We have 
\[ \operatorname{Ext}^1_{\Gamma}(\ZZ_{(p)}[\Lambda'], \ZZ_{(p)}) = (0) \]
by \cite[Key Lemma 2.1(i)]{CTS} together with the Change of Rings Theorem \cite[8.16]{CR}; therefore this sequence splits. In other words, 
there exists a $\ZZ_{(p)}\Gamma$-module homomorphism
$ f \colon \ZZ_{(p)}[\Lambda] \to \ZZ_{(p)}[\Lambda]$ 
such that the image of $f$
is $\langle m \rangle$ and $f(m) = m$. Then we can define $c_{\lambda} \in \ZZ_{(p)}$ by
$f(\lambda) = c_{\lambda} m$.
Note that $f(\gamma(\lambda)) = f(\lambda)$ and thus  
\begin{equation} \label{e.invariance}
c_{\gamma(\lambda)} = c_{\lambda}
\end{equation}
for every $\lambda \in \Lambda$ and $\gamma \in \Gamma$.
If $m = \sum_{\lambda \in \Lambda} a_{\lambda} \lambda$, as 
in the statement of the theorem, then $f(m) = m$ translates into
\[ \sum_{\lambda \in \Lambda} c_{\lambda} a_{\lambda} = 1 \, . \]
Since every $\Gamma$-orbit in $\Lambda$ has a power of $p$ elements,
reducing modulo $p$, we obtain
\[ \sum_{\lambda \in \Lambda^{\Gamma}} c_{\lambda} a_{\lambda} = 1 \pmod{p}
 \, . \]
This shows that $a_{\lambda} \ne 0$ modulo $p$,
for some $\lambda \in \Lambda^{\Gamma}$, as claimed.

\ref{ed1c} Decompose $X(T)_{(p)}$ uniquely into a direct sum of
indecomposable $\ZZ_{(p)} \Gamma$-lattices by the Krull-Schmidt
theorem \cite[Theorem 36.1]{CR}. Since $\ed(T;p)=1$, and the
essential $p$-dimension of tori is additive (Thm.\
\ref{thm.additive}), all but one of these summands are permutation
$\ZZ_{(p)} \Gamma$-lattices. Now by \cite[31.12]{CR}, we can lift
this decomposition to $X(T) \cong X(T') \oplus X(S)$, where
$\ed(T';p)=1$ and $\ed(S;p)=0$.
\end{proof}

\begin{exa}
Let $E$ be an \'etale algebra over $k$. It can be written as
$E=L_1\times\cdots\times L_m$ with some separable field extensions
$L_i/k$. The kernel of the norm $\R_{E/k}(\Gm)\to \Gm$ is denoted by
$\R_{E/k}^{(1)}(\Gm)$. It is a torus with lattice
\[ \bigoplus_{i=1}^m\ZZ[\Gamma/\Gamma_{L_i}]\;/\;\langle 1, \cdots , 1\rangle,\]
where $\Gamma=\Gal(\sep{k}/k)$ and $\Gamma_{L_i}=\Gal(\sep{k}/L_i)$.
Let $\Lambda$ be the disjoint union of the cosets $\Gamma/\Gamma_{L_i}$
Passing to a $p$-closure $\pcl{k}$ of $k$, $\Gamma_{\pcl{k}}$ 
fixes a $\lambda$ in $\Lambda$ iff $[L_i:k]$ is prime to $p$ for some $i$.
We thus have
\[ \ed_{k}(\R_{E/k}^{(1)}(\Gm);p)=\left\{
\begin{array}{ll}
1, &\mbox{$[L_i:k]$ is divisible by $p$ for all $i=1,...,m$}\\
0, &\mbox{$[L_i:k]$ is prime to $p$ for some $i$.}
\end{array}\right. \]
\end{exa}

\section{Tori split by cyclic extensions of degree dividing $p^2$}
\label{sect.splitsquare}
In this section we assume $k=\pcl{k}$ is $p$-closed.
Over $k=\pcl{k}$ every torus is split by a Galois extension of $p$-power order. We wish to compute the essential dimension of all tori split by a Galois extension with a (small) fixed Galois group $G$. 
The following theorem tells us for which $G$ this is feasible:
\begin{thm}[A.~Jones~\cite{Jo}]
\label{thm.Jones}
For a $p$-group $G$ there are only finitely many genera of indecomposable $\ZZ G$-lattices if and only if $G$ is cyclic of order dividing $p^2$. 
\end{thm}
\begin{rem}
\label{rem.Jones}
For $G=C_2\times C_2$ a classification of the (infinitely many) different genera of $\ZZ G$-lattices has been worked out by \cite{Na}. In contrast for $G=C_{p^3}$ or $G=C_p\times C_p$ and $p$ odd (in the latter case) no classification is known.
\end{rem}
Hence in this section we consider tori $T$ whose minimal splitting field is cyclic of degree dividing $p^2$. 
Its character lattice $X(T)$ is then a $\ZZ G$-lattice where
$G=\langle g|g^{p^2}=1\rangle$ denotes the cylic group of order $p^2$. 
Heller and Reiner \cite{HR}, (see also \cite[34.32]{CR}) classified all indecomposable $\ZZ G$-lattices. Our goal consists in computing the essential dimension of $T$. 
By Corollary \ref{cor.absoluteED} we have $\ed_k(T)=\ed_k(T;p)$, hence by the additivity Theorem~\ref{thm.additive} it will be enough to find the essential $p$-dimension of the tori corresponding to indecomposable $\ZZ G$-lattices.
Recall that two lattices are in the same genus if their $p$-localization (or equivalently $p$-adic completion) are isomorphic.
By Proposition~\ref{prop.equalED} tori with character lattices in the same genus have the same essential $p$-dimension, which reduces the task to calculating the essential $p$-dimension of tori corresponding to the $4p+1$ cases in the list \cite[34.32]{CR}.

Denote by $H=\langle\nolinebreak h|h^p=1\nolinebreak\rangle$ the group of order $p$.
We can consider $\ZZ H$ as a $G$-lattice with the action $g\cdot h^i=h^{i+1}$.
%We also denote by $\ZZ y\simeq \ZZ$ the trivial $G$-lattice with generator $y$.
Let 
\[
\delta_G=1+g+\ldots+g^{p^2-1}\quad\delta_H=1+h+\ldots+h^{p-1}
\] 
be the ``diagonals" in $\ZZ G$ and $\ZZ H$ and 
\[
\epsilon=1+g^p+\ldots+g^{p^2-p}.
\] 

%\[
%\delta_G=\delta_H+g\delta_H+\ldots+g^{p-1}\delta_H
%\]
%the ``diagonal" in $\ZZ G$.
The following $\ZZ G$-lattices represent all genera of indecomposable $\ZZ G$-lattices
(by $\langle * \rangle$ we mean the $\ZZ G$-sublattice generated by $*$):
\[
\everymath{\displaystyle}
\begin{array}{lcll}
M_1&=&\ZZ&\\
M_2&=&\ZZ H&\\ %(M_4)
M_3&=&\ZZ H/\langle\delta_H\rangle\\ %&=\ZZ H/\ZZ\delta_H \simeq \ZZ[\zeta_p]\\%(M_2)
M_4&=&\ZZ G\\%(M_{6,0})
M_5&=&\ZZ G /\langle\delta_G\rangle\\ %& =\ZZ G/\ZZ\delta_G	\\%(M_{8,0})
M_6&=&\ZZ G\oplus \ZZ / \langle\delta_G-p\rangle&\\%(M_{9,0})
M_7&=&\ZZ G/\langle\epsilon\rangle\\ %&= \ZZ G/\ZZ H\simeq\ZZ[\zeta_{p^2}]\\%(M_3)
M_8&=&\ZZ G/\langle\epsilon-g\epsilon\rangle&\\%(M_5)
M_{9,r}&=&\ZZ G\oplus \ZZ H /\langle\epsilon-(1-h)^r\rangle&1\le r\le p-1\\%(M_{6,r})
M_{10,r}&=&\ZZ G\oplus \ZZ H /\langle\epsilon(1-g)-(1-h)^{r+1}\rangle&1\le r\le p-2\\%(M_{7,r})
%\mbox{ or }M_{10,r}&=&\;\ZZ G\oplus \ZZ H \oplus \ZZ x /\langle \epsilon-(1-h)^r-x\rangle&\\
M_{11,r}&=&\ZZ G\oplus \ZZ H /\langle\epsilon-(1-h)^r,\delta_H\rangle&1\le r\le p-2\\%(M_{8,r})
M_{12,r}&=&\ZZ G\oplus \ZZ H / \langle\epsilon(1-g)-(1-h)^{r+1},\delta_H\rangle&1\le r\le p-2\\%(M_{9,r})
%\mbox{ or }M_{12,r}&=&\;\ZZ G\oplus \ZZ H \oplus \ZZ x / \langle\epsilon-(1-h)^r-x,\delta_H\rangle&
\end{array}
\]
In the sequel we will refer to the above list as $\LL$. \par
In $\LL$ we describe $\ZZ G$-lattices as quotients of permutation lattices of minimal possible rank, whereas \cite[34.32]{CR} describes these lattices as certain extensions $1 \to L \to M \to N \to 1$ of $\ZZ[\zeta_{p^2}]$-lattices by $\ZZ H$-lattices. 
Therefore these two lists look differently. Nevertheless they represent the same $\ZZ G$-lattices.
We show in the example of the lattice $M_{10,r}$ how one can translate from one list to the other.
%Let $\alpha =\epsilon(1-g)-(1-h)^{r+1}=(1-g)\cdot(\epsilon-(1-h)^r)$.
%$[\epsilon-(1-h)^r]-g\cdot[\epsilon-(1-h)^r]$.
%The submodule $N=\langle\alpha\rangle$ is generated as a $\ZZ$-module by elements $g^i\cdot\alpha$.
%It is easy to see that $g^p\cdot\alpha=\alpha$ and $(1+g+\ldots g^{p-1})\cdot\alpha=0$, so that $N$ is of rank $p-1$ and 
%$M_{10,r}$ of rank $p^2+1$.

%Introducing an extra summand $\ZZ x$ of rank 1 (with trivial $G$-action), the relations
%$g^i\cdot(\epsilon-(1-h)^r)=g^{i+1}\cdot(\epsilon-(1-h)^r)$ become $g^i\cdot(\epsilon-(1-h)^r)=x$, $i=0,\ldots,p^2-1$, and we can write
%\[
%M_{10,r}\simeq \ZZ G\oplus \ZZ H \oplus \ZZ x /\langle \epsilon-(1-h)^r-x\rangle
%\]
Let $\ZZ x$ be a $\ZZ G$-module of rank $1$ with trivial $G$-action. We have an isomorphism
\[M_{10,r}=\ZZ G \oplus \ZZ H/\langle \epsilon (1-g) - (1-h)^{r+1} \rangle \simeq \ZZ G\oplus \ZZ H \oplus \ZZ x /\langle \epsilon-(1-h)^r-x\rangle
\] induced by the inclusion $\ZZ G \oplus \ZZ H \hookrightarrow \ZZ G \oplus \ZZ H \oplus \ZZ x$.

This allows us to write $M_{10,r}$ as the pushout
\[
\xymatrix{
\ZZ H\ar@{->}[r]^{h\mapsto \epsilon}\ar@{->}[d]_{h\mapsto (1-h)^r+x}&\ZZ G\ar@{->}[d]\\
\ZZ H\oplus \ZZ x\ar@{->}[r]& M_{10,r}
}
\]
Completing both lines on the right we see that $M_{10,r}$ is an extension
\[
0\to \ZZ H\oplus \ZZ x\to M_{10,r} \to \ZZ G/\ZZ H \to 0
\]
with extension class determined by the vertical 
map $h\mapsto (1-h)^r+x$ cf. \cite[8.12]{CR} and 
we identify (the $p$-adic completion of) $M_{10,r}$ with 
one of the indecomposable lattices in the list \cite[34.32]{CR}. \par
Similarly, $M_1,\ldots,M_{12,r}$ are representatives 
of the genera of indecomposable $\ZZ G$-lattices. 
%(Note that $M_4=M_{9,0}$, $M_5=M_{11,0}$, $M_6=M_{12,0}$)
%Moreover the rank of $M_{10,r}$ is equal to $\rk \ZZ H \oplus \ZZ x + \rk \ZZ G/ \ZZ H = (p+1)+(p^2-p)=p^2+1$.
\begin{thm}
\label{prop.splitsquare}
Every indecomposable torus $T$ over $k$ split by $G$ has character 
lattice isomorphic to one of the $\ZZ G$-lattices $M$ in 
the list $\LL$ after $p$-localization and $\ed(T)=\ed(T;p)=\ed(\Diag(M);p)$.
Their essential dimensions are given in the tables below.
\[
\begin{array}{c|c|c}
M &\rk M&\ed(T)\\
\hline
M_1&1&0\\
M_2&p&0\\
M_3&p-1&1\\
M_4&p^2&0\\
M_5&p^2-1&1\\
M_6&p^2&1
\end{array}
\quad\quad
\begin{array}{c|c|c}
M &\rk M&\ed(T)\\
\hline
M_7&p^2-p&p\\
M_8&p^2-p+1&p-1\\
M_{9,r}&p^2&p\\
M_{10,r}&p^2+1&p-1\\
M_{11,r}&p^2-1&p+1\\
M_{12,r}&p^2&p
\end{array}
\]
\end{thm}
\begin{proof}[{Proof of Proposition \ref{prop.splitsquare}}]
%The rank of the lattices is computed by writing out the relations as shown for the lattice $M_{10,r}$.
We will assume $p>2$ in the sequel.
For $p=2$ the Theoerem is still true but some easy additional arguments are needed which we leave out here.

The essential $p$-dimension of tori corresponding to $M_1\ldots,M_6$ easily follows from the discussion in 
section~\ref{sect.ed<=1}.
Let $M$ be one of the lattices $M_7,\ldots,M_{12,r}$ and $T=\Diag M$ the corresponding torus.
We will determine the minimal rank of a permutation $\ZZ G$-lattice $P$ admitting a homomorphism $P \to M$ which becomes surjective after localization at $p$.
Then we conclude $\ed(T;p)=\rk P-\rk M$ with Corollary~\ref{cor.lattice}.

We have the bounds
\begin{equation}
\rk M\le \rk P\le p^2 \;(\mbox{or }p^2+p),
\end{equation}
where the upper bound holds since every $M$ is given as a quotient of $\ZZ G$ (or $\ZZ G \oplus \ZZ H$).
Let $C=\Split (T[p])$ the finite constant group used in the proof of Theorem~\ref{thm:MainTheorem1}. 
The rank of $C$ determines exactly the number of direct summands into which $P$ decomposes.
Moreover each indecomposable summand has rank a power of $p$.

As an example, we show how to find $C$ for $M=M_{11,r}$:
The relations $g^j\cdot(\epsilon-(1-h)^r);\delta_H$ are written out as
\[
\sum_{i=0}^{p-1}g^{pi+j}-\sum_{\ell=0}^r\binom{r}{\ell}(-1)^\ell h^{\ell+j},\; 0\leq j \leq p-1;\quad \sum_{i=0}^{p-1}h^i
\]
and the $\sep{k}$-point of the torus are
\[
\everymath{\displaystyle}
\begin{array}{ll}
T(\sep{k})=&{\Big\{} (t_0,\dotsc,t_{p^2-1},s_0,\dotsc,s_{p-1})\mid\\
&\prod_{i=0}^{p-1} t_{pi+j}=\prod_{\ell=0}^r s_{\ell+j}^{(-1)^\ell \binom{r}{\ell}},\; 0\leq j \leq p-1;\quad\prod_{i=0}^{p-1} s_i=1 {\Big\}} 
\end{array}
\]
and $C$ is the constant group of fixed points of the $p$-torsion $T[p]$:
\[
C(k)=\left\{\left(\zeta_p^i,\dotsc,\zeta_p^i,\zeta_p^j,\dotsc,\zeta_p^j\right)\mid\ 0\leq i,j \leq p-1\right\} \simeq \mu_p^2.
\]
(Note that the primitive $p$th root of unity $\zeta_p$ is in $k$ by our assumption that $k$ is $p$-closed).
For other lattices this is similar: $C$ is equal to $\Split(\Diag(P)[p])\simeq \mu_p^r$ where $M$ is presented as a quotient $P/N$ of a permutation lattice $P$ (of minimal rank) as in $\LL$ and where $r$ denotes the number of summands in a decomposition of $P$. 
\[
\begin{array}{c|c|c|c}
M &\mbox{rank $C$}&\mbox{rank $M$}&\mbox{possible $\rk P$}\\
\hline
M_7&1&p^2-p&p^2\\
M_8&1&p^2-p+1&p^2\\
M_{9,r}&2&p^2&p^2+1\mbox{ or }p^2+p\\
M_{10,r}&2&p^2+1&p^2+1\mbox{ or }p^2+p\\
M_{11,r}&2&p^2-1&p^2+1\mbox{ or }p^2+p\\
M_{12,r}&2&p^2&p^2+1\mbox{ or }p^2+p
\end{array}
\]
We need to exclude the possibility $\rk P=p^2+1$ for the lattices $M=M_{9,r},\ldots,M_{12,r}$.
We can only have the value $p^2+1$ if there exists a character in $M$ which is fixed under the Galois group and nontrivial on $C$. 
The following Lemma \ref{lem:fixedCharacters} tells us, that such characters do not exist in either 
case. 
Hence the minimal dimension of a $p$-faithful representation of all these tori is $p^2+p$.
\end{proof}

\begin{lem}
\label{lem:fixedCharacters}
For $i=9,\ldots,12$ and $r\geq 1$ every character $\chi \in M_{i,r}$ fixed under $G$ has trivial restriction to $C$.
\end{lem}
\begin{proof}
By \cite{Hi} the cohomology group $H^0(G,M_{i,r})=M_{i,r}^G$ of $G$-fixed points in $M_{i,r}$ is trivial for $i=11$, has rank $1$ for $i=9,12$ and rank $2$ for $i=10$, respectively. 
They are represented by $\ZZ \delta_H$ in $M_{9,r}$, by $\ZZ (\epsilon - (1-h)^r)$ in $M_{12,r}$ and by $\ZZ (\epsilon - (1-h)^r)\oplus \ZZ\delta_H$ in $M_{10,r}$, respectively. 
Since all these characters are trivial on \[
C=\Split(\Diag(\ZZ G\oplus \ZZ H)[p]),
\]
the claim follows.
\end{proof}

\section*{Acknowledgments}
The authors are grateful to A. Auel, A. Merkurjev and A. Vistoli for helpful comments and conversations.


\begin{thebibliography}{99}

\bibitem[AP]{AP}
H.~Abold, W.~Plesken, {\em Ein Sylowsatz f\"ur endliche $p$-Untergruppen 
von ${\rm GL}(n,Z)$}, Math. Ann.  232  (1978), no. 2, 183--186.

\bibitem[Ba]{Ba} G. Bayarmagnai, {\em
Essential dimension of some twists of $\mu_{p^n}$},
Proceedings of the Symposium on Algebraic Number Theory
and Related Topics,  145--151,
RIMS K$\hat{\rm o}$ky$\hat{\rm u}$roku Bessatsu,
B4, Res. Inst. Math. Sci. (RIMS), Kyoto (2007).

\bibitem[BF]{BF}
G.~Berhuy, G.~Favi, {\em Essential Dimension: A Functorial Point of View (after A.~Merkurjev)}, Doc.~Math. 8:279--330 (electronic) (2003).

\bibitem[Bo]{Bo} 
A.~Borel {\em Linear Algebraic Groups}, Benjamin (1969).

\bibitem[BS]{bs} 
A. Borel, J.-P. Serre, {\em Th\'eor\`emes de finitude 
en cohomologie galoisienne}, Comment. Math. Helv. {\bf 39} (1964), 
111--164.

\bibitem[Bou]{Bou}
N.~Bourbaki, {\em Algebra. II. Chapters 4--7}.
Translated from the French by P. M. Cohn and J. Howie.
Elements of Mathematics. Springer-Verlag, Berlin,  (1990).

\bibitem[BR]{BR}
J.~Buhler, Z.~Reichstein, {\em On the essential dimension of a finite group}, Compositio Mathematica 106:159--179.(1997).

\bibitem[CGR]{cgr}
V. Chernousov, Ph. Gille, Z. Reichstein, {\em Resolving 
$G$-torsors by abelian base extensions},  J. Algebra  {\bf 296} (2006),  
no. 2, 561--581.

\bibitem[CTS]{CTS}
J.-L.~Colliot-Th\'{e}l\`{e}ne, J.~J. Sansuc, {\em Principal
Homogeneous Spaces under Flasque Tori: Applications}, 
J. Algebra {\bf 106} (1087), 148--205.

\bibitem[CR]{CR} C.~W.~Curtis, I.~Reiner, {\em Methods
of representation theory} ,vol. 1, Wiley (Interscience), 1981.

\bibitem[DG]{DG} M.~Demazure, P.~Gabriel, {\em Groupes alg\'ebriques. 
Tome I}, Masson \& Cie, Paris; North-Holland Publishing Co., Amsterdam, 1970.

\bibitem[Fl]{Fl}
M.~Florence, {\em On the essential dimension of cyclic $p$-groups}, Inventiones Mathematicae, {\bf 171} (2007), 175-189. 

\bibitem[GMS]{gms}
S. R.~Garibaldi, A.~Merkurjev, J.-P.~Serre:
{\em Cohomological Invariants in Galois
Cohomology},  University Lecture Series, Vol. 28,
American Mathematical Society, Providence, RI, (2003).

\bibitem[GR]{GR}
Ph.~Gille, Z.~Reichstein, {\em A lower bound on 
the essential dimension of a connected linear group},
Comment. Math. Helv.  {\bf 4},  no. 1 (2009), 189--212.

\bibitem[Gro]{grothendieck}A.~Grothendieck,
{\em La torsion homologique et
les sections rationnelles}, Expos\'{e} 5, S\'eminaire C.~Chevalley,
Anneaux de Chow et applications, IHP, (1958).

\bibitem[HR]{HR}
A. Heller, I. Reiner: {\em Representations of cyclic groups
in rings of integers I},  Annals of Math, {\bf 76} (1962), 73-92.

\bibitem[Hi]{Hi}
H.~Hiller, {\em Flat Manifolds with $\ZZ/p^2$ Holonomy}, 
L'Enseignement Math\'{e}matique, {\bf 31} (1985), 283--297.

\bibitem[Ja]{Ja}
J.~C.~Jantzen, {\em Representations of Algebraic Groups}. Pure 
and Applied Mathematics, 131. Academic Press, Orlando, Florida, (1987).

\bibitem[Jo]{Jo}
A.~Jones, \textit{Groups with a finite number of indecomposable integral representations}, Mich.~Math.~J, {\bf 10} (1963), 257-261.

\bibitem[Ka]{Ka}
G.~Karpilovsky, {\em Clifford Theory for Group Representations}. Mathematics Studies, 156. North-Holland, Netherlands, (1989).

\bibitem[KM]{KM}
N.~Karpenko, A.~Merkurjev, {\em Essential dimension of finite p-groups},
Inventiones Mathematicae, {\bf 172} (2008), 491--508.

%\bibitem[KMRT]{KMRT}
%M.-A.~Knus, A.~Merkurjev, M.~Rost, J.-P.~Tignol, {\em The book
%of involutions}.  AMS Colloquium Publications, 44. AMS, Providence, RI, 
%(1998).

% \bibitem[LP]{LP}
% C. R.~Leedham-Green, W.~Plesken, {\em Some remarks 
% on Sylow subgroups of general linear groups},
% Math. Z. 191 no. 4, (1986), 529--535.

\bibitem[Ma]{Ma}
B.~Margaux, {\em Passage to the limit in non-abelian \v{C}ech cohomology}.
J. Lie Theory  17,  no. 3 (2007), 591--596.

\bibitem[Me$_1$]{Me1}
A.~Merkurjev, {\em Essential dimension}, in
Quadratic forms -- algebra, arithmetic, and geometry (R. Baeza, W.K. Chan,
D.W. Hoffmann, and R. Schulze-Pillot, eds.), Contemporary Mathematics 
{\bf 493} (2009), 299--326.

\bibitem[Me$_2$]{Me2}
A.~Merkurjev \textit{Essential dimension of PGL($p^2$)}, preprint,
available at \url{http://www.math.ucla.edu/~merkurev/publicat.htm}.

\bibitem[MR$_1$]{mr1}
A.~Meyer, Z.~Reichstein,
{\em The essential dimension of the normalizer 
of a maximal torus in the projective linear group},
Algebra and Number Theory, {\bf 3}, no. 4 (2009),
467--487.

\bibitem[MR$_2$]{mr2}
A.~Meyer,  Z.~Reichstein, {\em An upper bound on the essential 
dimension of a central simple algebra}, to appear in 
Journal of Algebra, 10.1016/j.jalgebra.2009.09.019, preprint
available at arXiv:0907.4496

\bibitem[NA]{Na} 
L.~A.~Nazarova, {\em Unimodular representations 
of the four group}, Dokl. Akad. Nauk SSSR, {\bf 140} (1961), 1011--1014. 

%\bibitem[PV]{PV94} 
%V.L.~Popov, E.B.~Vinberg, {\em Invariant Theory}, in Encyclopedia of Math. Sciences,
%55, Algebraic Geometry IV, edited by A.N. Parshin and I.R. Shafarevich,
%Springer-Verlag, (1994).

\bibitem[Re]{Re} Z.~Reichstein, {\em On the Notion of Essential Dimension 
for Algebraic Groups}, Transformation Groups, {\bf 5}, 3 (2000), 265-304.

\bibitem[RY]{RY} Z.~Reichstein, B.~Youssin, {\em Essential Dimensions 
of Algebraic Groups and a Resolution Theorem for $G$-varieties}, 
with an appendix by  J.~Kollar and E.~Szabo, Canadian Journal 
of Mathematics, {\bf 52}, 5 (2000), 1018--1056.

%\bibitem[Rein]{Rein} I.~Reiner, {\em A survey on Integral Representation Theory}, Bull. Amer. Math. Soc,  {\bf 76}, 2 (2970), 159--227.

\bibitem[RZ]{RZ}
L.~Ribes, P.~ Zalesskii, {\em Profinite Groups}.
Springer-Verlag, Berlin, 2000.

\bibitem[Ro]{Ro} M.~Rost, {\em Essential dimension of twisted $C_4$},
available at \url{http://www.math.uni-bielefeld.de/~rost/ed.html}.

\bibitem[Sch$_1$]{Sc1} H.-J.~Schneider, {\em Zerlegbare 
Erweiterungen affiner Gruppen} J. Algebra {\bf 66}, no. 2 (1980),
569--593.

\bibitem[Sch$_2$]{Sch} H.-J.~Schneider, {\em Decomposable Extensions 
of Affine Groups}, in Lecture Notes in Mathematics {\bf 795}, 
Springer Berlin/Heidelberg (1980), 98--115.

\bibitem[Sch$_3$]{Sc2} H.-J.~Schneider, {\em Restriktion und 
Corestriktion für algebraische Gruppen} J. Algebra, {\bf 68},
no. 1 (1981), 177--189.

\bibitem[Se$_1$]{serre-rep}
J.-P. Serre, {\em Linear representations of finite groups},
Graduate Texts in Mathematics, {\bf 42}, Springer-Verlag, 1977.

\bibitem[Se$_2$]{serre-gc}
J.-P. Serre, {\em Galois cohomology}.
Springer Monographs in Mathematics.
Springer-Verlag, Berlin, 2002.

\bibitem[Ta]{Ta}
J. Tate, {\em Finite flat group schemes}.
Modular forms and Fermat's last theorem (Boston, MA, 1995),
121--154, Springer, New York, 1997.

\bibitem[Vo]{voskresenskii}
V.~E.~Voskresenski{\u\i}, {\em Algebraic
Groups and Their Birational Invariants}, American
Mathematical Society, Providence, RI, 1998.

\bibitem[Wa]{Wa}
W.~C.~Waterhouse, {\em Introduction to affine group schemes}.
Springer-Verlag, New York-Berlin, 1979.

\bibitem[Wi]{Wi}
J.~S.~Wilson, {\em Profinite Groups}.
London Math. Soc. Monographs 19,
Oxford University Press, New York, 1998.

\bibitem[Win]{Win}
D.~Winter, {\em The structure of fields}. Graduate Texts in Mathematics, 
no. 16. Springer-Verlag, New York-Heidelberg, 1974.

\end{thebibliography}
\end{document}